\documentclass[a4paper,11pt]{article}
\usepackage[utf8]{inputenc}
\usepackage{a4wide}
\usepackage{latexsym,amsfonts,amsmath,mathrsfs,url,amsthm}
\usepackage{mathtools}
\usepackage{color,graphicx}
\usepackage{lipsum}
\usepackage{hyperref}
\usepackage{xcolor}
\usepackage{cleveref}
\usepackage{amssymb}
\usepackage{authblk}

\newtheorem{theorem}{Theorem}[section]
\newtheorem{lemma}[theorem]{Lemma}

\newtheorem{algorithm}[theorem]{Algorithm}
\newtheorem{remark}[theorem]{Remark}

\newtheorem{corollary}[theorem]{Corollary}
\providecommand{\keywords}[1]
{
  \noindent \small	
  \textbf{Keywords:} #1
}
\providecommand{\amscode}[1]
{
  \noindent \small	
  \textbf{AMS subject classifications:} #1
}
\makeatletter
\def\blfootnote#1{%
  \gdef\@thefnmark{}%
  \@footnotetext{\noindent #1}%
}
\makeatother
\newcommand{\rd}{\, \mathrm{d}}
\newcommand{\rand}{\mathrm{rand}}

\newcommand{\bszero}{\boldsymbol{0}}
\newcommand{\bsh}{\boldsymbol{h}}

\newcommand{\bsk}{\boldsymbol{k}}
\newcommand{\bsl}{\boldsymbol{\ell}}
\newcommand{\bsL}{\boldsymbol{L}}
\newcommand{\bsx}{\boldsymbol{x}}

\newcommand{\bsz}{\boldsymbol{z}}
\newcommand{\bsgamma}{\boldsymbol{\gamma}}
\newcommand{\bsDelta}{\boldsymbol{\Delta}}
\newcommand{\EE}{\mathbb{E}}

\newcommand{\NN}{\mathbb{N}}

\newcommand{\RR}{\mathbb{R}}
\newcommand{\ZZ}{\mathbb{Z}}
\newcommand{\Acal}{\mathcal{A}}
\newcommand{\Bcal}{\mathcal{B}}
\newcommand{\Hcal}{\mathcal{H}}
\newcommand{\Pcal}{\mathcal{P}}

\newcommand{\fraku}{\mathfrak{u}}

\newcommand{\tmod}[1]{{\;(\mathrm{mod}\; #1)}}

\newcommand{\e}{\mathbb{E}}
\newcommand{\bsone}{\boldsymbol{1}}

\DeclareMathOperator{\supp}{supp}

\newcommand{\icomp}{\mathtt{i}}
\newcommand{\abs}[1]{\left\vert#1\right\vert}
\newcommand{\norm}[1]{\left\Vert#1\right\Vert}
\allowdisplaybreaks

\title{Universal $L_2$-approximation using median lattice algorithms}
\author[1,2]{Zexin Pan}
\author[3]{Takashi Goda}
\author[2]{Peter Kritzer}

\affil[1]{Institute of Fundamental and Transdisciplinary Research, Zhejiang University, 866 Yuhangtang Road, Xihu District, Hangzhou, Zhejiang Province, 310058, China}
\affil[2]{Johann Radon Institute for Computational and Applied Mathematics (RICAM), Austrian Academy of Sciences, Altenbergerstr. 69, 4040 Linz, Austria}
\affil[3]{Graduate School of Engineering, The University of Tokyo, 7-3-1 Hongo, Bunkyo-ku, Tokyo 113-8656, Japan}

\date{\today}

\begin{document}

\maketitle

\blfootnote{Email addresses: \url{zep002@zju.edu.cn} (Zexin Pan) ; \url{goda@frcer.t.u-tokyo.ac.jp} (Takashi Goda); \url{peter.kritzer@oeaw.ac.at} (Peter Kritzer)}

\sloppy

\begin{abstract}
We study the problem of multivariate $L_2$-approximation of functions in a weighted Korobov space using a median lattice-based algorithm recently proposed by the authors. In the original work, the algorithm requires knowledge of the smoothness and weights of the Korobov space to construct the hyperbolic cross index set, where each coefficient is estimated via the median of approximations obtained from randomly shifted, randomly chosen rank-1 lattice rules. In this paper, we introduce a \emph{universal median lattice-based algorithm}, which eliminates the need for any prior information on smoothness and weights. Although the tractability property of the algorithm slightly deteriorates, we prove that, for individual functions in the Korobov space with arbitrary smoothness and (downward-closed) weights, it achieves an $L_2$-approximation error arbitrarily close to the optimal rate with respect to the number of function evaluations. Numerical experiments are conducted to support our theoretical claim.
\end{abstract}


\amscode{41A25, 41A63, 65D15, 65D30, 65Y20}
%%%%%%%%%%%%%%%%%%%%%%%%%%%%%%%%%%

%%%%%%%%%%%%%%%%%%%%%%%%%%%%%%%%%%%%%%%%%%

\section{Introduction}

This paper deals with the classical problem of approximating functions in the $L_2$-norm, defined on a so-called \textit{weighted Korobov space}, which consists of $d$-variate, one-periodic functions with sufficiently fast decaying Fourier coefficients. Korobov spaces are reproducing kernel Hilbert spaces, and it is known that their elements can be represented pointwise by absolutely convergent Fourier series, i.e., we are dealing with functions of the form 
\[
f(\bsx)=\sum_{\bsh\in\ZZ^d} \widehat{f}(\bsh) \exp (2\pi\icomp \bsh\cdot \bsx),
\]
where, as usual, 
\[
 \widehat{f}({\bsh}):=\int_{[0,1)^d} f(\bsx) \exp({-2\pi\icomp \bsh\cdot \bsx}) \rd \bsx
\]
is the $\bsh$-th Fourier coefficient of $f$ for $\bsh\in\ZZ^d$, where $\ZZ$ denotes the set of integers. 

Korobov spaces are
of particular interest in the context of \textit{quasi-Monte Carlo (QMC) rules}, and have been studied frequently in the literature, see, e.g., \cite{dick2022lattice, kuo2013high, novak2008tractability} and the references therein. Indeed, we shall also use QMC methods in the present paper; these methods are equal-weight quadrature rules that are frequently used for the numerical integration of multivariate functions. To be more precise, QMC rules have the form 
\[
\frac{1}{N}\sum_{k=0}^{N-1}g(\bsx_k) \approx \int_{[0,1]^d} g(\bsx) \rd \bsx
\]
for a suitably defined $g$, where $\bsx_0,\ldots,\bsx_{N-1}$ are carefully selected sampling points. 

In the context of 
$L_2$-approximation, QMC rules are typically employed for approximately computing suitably selected Fourier coefficients of the function to be approximated, or in splines. Previous work analyzing such methods include \cite{BGKS24,CGK24,CKNS20,kuo2006lattice,li2003trigonometric,PKG24, zeng2006error} and many others. The approach considered here uses \textit{rank-1 lattice point sets} as the QMC samples, and in particular studies the use of so-called \textit{median rules}. Median rules have first been dealt with for numerical integration and are based on the following principle. When working with QMC integration rules, one often introduces a suitable randomization thereof, and then repeats the integration algorithm for independent realizations of the randomized rule. For median rules, one then takes the median of these independent instances as an approximation to an integral. A considerable advantage of median rules over usual QMC rules is that they are frequently independent of the parameters (as, e.g., the smoothness) of particular function spaces, and therefore offer a high degree of \textit{universality}. Median rules for numerical integration have been presented, e.g., in \cite{goda2022free,goda2024universal,pan2024super-pol}. In the recent paper \cite{PKG24}, the authors have successfully used median lattice algorithms for $L_2$-approximation in weighted Korobov spaces. However, a drawback of the results in that paper is the partial loss of the universality of the median approach, due to the definition of the algorithm in \cite{PKG24}. The goal of the present paper is to provide an alternative algorithm that regains this universality. As we shall see below, our algorithm can be shown to converge at a rate that is arbitrarily close to optimal. By ``arbitrarily close to optimal'' we mean that if we consider the number $M$ of function evaluations of a given function $f$ required by our algorithm to compute an approximation to $f$, the algorithm converges with a rate of $\mathcal{O} (M^{-\alpha + \varepsilon})$ for arbitrarily small $\varepsilon>0$, where $\alpha$ is the \textit{smoothness parameter} of the Korobov space under consideration (see below for the precise definition of Korobov spaces).

As shown in \cite{kuo2006lattice} and related works, using a single lattice rule with $M$ nodes to approximate the Fourier coefficients yields a convergence rate of the worst-case (i.e., considering the supremum over all functions in the unit ball of the space) $L_2$-approximation error (see \Cref{sec:preliminaries}) arbitrarily close to $\mathcal{O}(M^{-\alpha/2})$. This result is in a certain sense best possible, because, in a slightly more general setting, \cite{byrenheid2017tight} proves that the rate $\mathcal{O}(M^{-\alpha/2})$ cannot be improved with such an approach for any $d\ge 2$. However, as shown in \cite{kaemmerer2019constructing} and \cite{kammerer2019approximation}, this convergence rate can be improved in a probabilistic sense by employing so-called \textit{multiple lattices}. Moreover, in the recent paper \cite{CGK24} by Cai, Goda, and Kazashi, the authors introduced a randomized lattice-based algorithm for $L_2$-approximation and proved that one can achieve a convergence rate arbitrarily close to $\mathcal{O}(M^{-\alpha(2\alpha+1)/(4\alpha+1)+\varepsilon})$ of the worst-case root-mean-squared $L_2$-approximation error, which is better than the rate $\mathcal{O}(M^{-\alpha/2})$. The same paper proves that this randomized algorithm cannot achieve a better rate than $\mathcal{O}(M^{-\alpha/2 - 1/2})$. In \cite{PKG24}, we used median lattice rules based on 
$R$ randomized lattice point sets with $N$ points each, i.e., our algorithm required $M=RN$ function values of a function $f$ to be approximated. For such algorithms, we could show that we achieve 
a root-mean-squared error with a convergence order arbitrarily close to $\mathcal{O} (M^{-\alpha})$, and this result is best possible for Korobov spaces. However, as pointed out above, the drawback of the result in \cite{PKG24} is that the algorithm is not universal with respect to the parameters of the Korobov space, which we try to overcome in the present paper. (Note that there exist other approaches by various authors that yield a convergence rate arbitrarily close to $\mathcal{O} (M^{-\alpha})$, but these are not based on lattice rules or QMC methods; see, e.g., \cite{CM17, DC24, K19, WW07}.)

Moreover, we will show that, under suitable assumptions, this convergence can take place with only a mild dependence on the dimension of the problem. We will achieve this by making use of the concept of weighted function spaces in the sense of Sloan and Wo\'{z}niakowski, see \cite{sloan1998when}. To this end, let $\bsgamma=(\gamma_{\fraku})_{\fraku\subseteq \{1{:}d\}}$ be a collection of \textit{weights}. Here and in the following, we write $\{1{:}V\}$ to denote the set $\{1,\ldots, V\}$ for a positive integer $V$. The basic idea of weights is to assign a positive number to every possible group of variables $\fraku\subseteq \{1{:}d\}$, which describes how much influence the
respective groups of variables have in the problem. Large values of $\gamma_{\fraku}$ indicate higher influence, low values of $\gamma_{\fraku}$ indicate less influence. These weights are incorporated in the inner product and norm of the weighted Korobov space. It is known that under certain summability conditions on the weights the curse of dimensionality can be weakened or even vanquished completely. This corresponds to the idea that, while the approximation problem may nominally depend on a large number of variables, their influence on the problem diminishes fast, such that only rather few of them contribute significantly to the problem. 

The rest of the paper is structured as follows.
In Section~\ref{sec:preliminaries}, we introduce the necessary preliminaries and notation, including the weighted Korobov space, the $L_2$-approximation problem, and rank-1 lattice rules.
Section~\ref{sec:median} presents the universal median lattice-based algorithm and states the main probabilistic $L_2$-error bound (Theorem~\ref{thm:combined_error}), along with a discussion of its tractability properties.
In Section~\ref{sec:error}, we provide a detailed analysis of the three terms in the error decomposition \eqref{eqn:errordecom}, establishing probabilistic bounds that, when combined, yield the main result.
In Section~\ref{sec:experiment}, we conduct numerical experimetns to support our theoretical claim.
We conclude this paper with a discussion in Section~\ref{sec:discussion}.

\section{Preliminaries and notation}\label{sec:preliminaries}

In what follows, we denote the set of integers by $\ZZ$ and the set of positive integers by $\NN$. As usual, we denote the exponential function by $\exp$, and write $\exp(1)=e$ for short. Bold symbols are used to denote vectors, whose lengths are usually clear from the context. We denote the imaginary unit by $\icomp := \sqrt{-1}$, and further notation will be introduced as necessary.

\subsection{The weighted Korobov space}
Let $\bsgamma=(\gamma_{\mathfrak{u}})_{\mathfrak{u}\subseteq\{1{:}d\}}$ be a collection of positive real numbers, which we refer to as the weights in what follows. As mentioned in the introduction, the weight $\gamma_{\mathfrak{u}}$ models the influence of the group of variables $x_j$ with indices $j\in\mathfrak{u}$. 
Throughout this paper, we assume $0<\gamma_\fraku\leq 1$; this assumption could be relaxed without changing the essence of our results, at the cost of additional technical notation.
We further assume that the weights satisfy $\gamma_{\mathfrak{v}} \le \gamma_{\mathfrak{w}}$ whenever $\mathfrak{w}\subseteq \mathfrak{v}$, a property we refer to as being \emph{downward-closed}.

For a collection of weights $\bsgamma=(\gamma_{\mathfrak{u}})_{\mathfrak{u}\subseteq\{1{:}d\}}$ and a real number $\alpha>1/2$, we define, for $\bsh=(h_1,\ldots,h_d)\in \ZZ^d$,
\[ 
r_{2\alpha,\bsgamma} (\bsh):=\gamma_{\supp(\bsh)}^{-1}\prod_{j\in \supp(\bsh)} \abs{h_j}^{2\alpha},
\]
where $\supp(\bsh)=\{j\in \{1{:}d\}\, \colon\, h_j\neq 0\}$.

We now define the weighted Korobov space $\Hcal_{d,\alpha,\bsgamma}$, which is a reproducing kernel Hilbert space consisting of one-periodic functions on $[0,1)^d$, where periodicity is understood with respect to each variable. The norm in $\Hcal_{d,\alpha,\bsgamma}$ is defined in terms of the Fourier coefficients by
\[
  \|f\|_{d,\alpha,\bsgamma}^2:=\sum_{\bsh\in\ZZ^d} \abs{\widehat{f}({\bsh})}^2 r_{2\alpha,\bsgamma}(\bsh).
\]
Note that $\Hcal_{d,\alpha,\bsgamma}$ is a subspace of $L_2 ([0,1)^d)$, and that all elements of the Korobov space with $\alpha>1/2$ can be represented pointwise by their Fourier series. Moreover, it is known that if $\alpha\in\NN$, the univariate Korobov (periodic Sobolev) space that serves as the building block of $\Hcal_{d,\alpha,\bsgamma}$ consists precisely of one-periodic functions on $[0,1)$ whose derivatives up to order $\alpha-1$ are absolutely continuous and whose $\alpha$-th derivative belongs to $L_2([0,1))$. For this reason, $\alpha$ is called the smoothness parameter. In the classical product-weight setting (i.e., when $\gamma_{\mathfrak{u}}=\prod_{j\in\mathfrak{u}}\gamma_j$), the $d$-dimensional space $\Hcal_{d,\alpha,\bsgamma}$ is the $d$-fold tensor product of these univariate spaces. For general weights $\bsgamma$, the space is defined by the Fourier-weighted norm above; see \cite{dick2022component} for further details.

\subsection{$L_2$-approximation}
We will study $L_2$-approximation of functions $f\in \Hcal_{d,\alpha,\bsgamma}$. More formally, this means that we study the approximation of the \textit{embedding
operator} $S:\Hcal_{d,\alpha,\bsgamma}\to L_2 ([0,1)^d)$, $S(f)=f$. For a given deterministic approximation algorithm $A:\Hcal_{d,\alpha,\bsgamma}\to L_2 ([0,1)^d)$, a common error measure considered in the literature is the worst-case error,
\[
  \mathrm{err}(\Hcal_{d,\alpha,\bsgamma}, L_2, A):= \sup_{\substack{f\in \Hcal_{d,\alpha,\bsgamma}\\ 
  \norm{f}_{d,\alpha,\bsgamma}\le 1}}
  \|f-A(f)\|_{L_2},
\]
i.e., one considers the worst performance of $A$ across the unit ball of $\Hcal_{d,\alpha,\bsgamma}$. 

In this paper, we will deal with \textit{randomized algorithms} instead of deterministic algorithms. When we speak of a randomized approximation algorithm, we refer to a pair consisting of a probability space $(\Omega, \Sigma, \mu)$ and a family of mappings $A = (A^\omega)_{\omega \in \Omega}$, where each $A^\omega$ is a deterministic approximation algorithm for fixed $\omega \in \Omega$. 
In the randomized setting, a common error measure is the worst-case root-mean-squared error, also referred to as the \emph{randomized error}, defined by
\begin{align*}
  \mathrm{err}^{\mathrm{ran}}(\Hcal_{d,\alpha,\bsgamma}, L_2, (A^{\omega})):= \sup_{\substack{f\in \Hcal_{d,\alpha,\bsgamma}\\ 
  \norm{f}_{d,\alpha,\bsgamma}\le 1}}
  \left(\EE_{\omega}\left[\|f-A^{\omega}(f)\|^2_{L_2}\right]\right)^{1/2}.
\end{align*}
Alternatively, one can consider the $(\epsilon,\delta)$-approximation framework, see for instance \cite{kunsch2019optimal}, which asks for which pairs $(\epsilon,\delta)$ the algorithm satisfies
\begin{align}\label{eq:eps-delta_framework} 
\mathrm{Pr}_{\omega}\left[\|f-A^{\omega}(f)\|_{L_2}>\epsilon \right]\le \delta \qquad \text{for all $f$ with $\norm{f}_{d,\alpha,\bsgamma}\le 1$.} 
\end{align}

Our main result (see Corollary~\ref{thm:combined_error} below) is that we can reach a probabilistic error bound of order $M^{-\alpha}$ when using $M$ function evaluations with high probability, by using a median algorithm (see Section \ref{sec:median} for the precise definition), where the median is taken over $R$ randomized instances of QMC rules. Here, $M$ is of the form $RN$, where $N$ is the number of integration nodes used in each single QMC rule. The essential progress in comparison to our previous result in \cite{PKG24} is that the algorithm will not be depending on the parameters $\alpha$ and $\bsgamma$ of $\Hcal_{d,\alpha,\bsgamma}$, and thus can be considered as much more universal. 

\subsection{Rank-1 lattice rules}
Let us now outline which type of QMC rules we use in our median algorithm, which are rank-1 lattice rules. For simplicity, let us assume that $N$ is a prime in the following (assuming that $N$ is prime guarantees that all one-dimensional projections of a lattice point set are evenly distributed, which is an advantage. If we would allow composite $N$, we presumably could derive similar results, but would need more technical notation). Let $\bsz=(z_1,\ldots,z_d)$ be a 
vector with each component in $\{1{:}(N-1)\}$. We can then define a 
lattice point set $\Pcal_N$ with points $\bsx_0,\ldots,\bsx_{N-1}$ as follows. For $j\in\{1{:}d\}$ and $k\in\{0,\ldots,N-1\}$, the $j$-th component of $\bsx_k$ is given by
\[
x_k^{(j)}:=\left\{\frac{k z_j}{N}\right\},
\]
where $\{y\}=y-\lfloor y \rfloor$ denotes the fractional part of a real number $y$. By repeating the above procedure for all $j\in\{1{:}d\}$ and $k\in\{0,\ldots,N-1\}$, we obtain the full lattice point set $\Pcal_N$. A QMC rule using $\Pcal_N$ is called a (rank-1) lattice rule. 

Note that, for fixed $N$ and $d$, a lattice rule is fully characterized by the choice of the \textit{generating vector} $\bsz$. Not all choices of $\bsz$ yield lattice rules that have sufficient quality to be usefully employed in integration or approximation algorithms. However, there are fast construction algorithms available that return good generating vectors for given $N$, $d$, $\alpha$, and $\bsgamma$. We refer to the books \cite{dick2022lattice, sloan1994lattice} for overviews of the theory of lattice rules, and in particular to \cite[Chapters 3 and 4]{dick2022lattice} for constructions of good lattice rules. In the present paper, however, we will make a random choice of the generating vectors $\bsz$ to obtain the lattice point sets used. Therefore, we need not be concerned with construction algorithms here, which is a general computational advantage of median rules. 

It is sometimes useful, and we shall also do so here, to introduce an additional random element when applying lattice rules. This is commonly achieved by a \emph{random shift}, $\Delta\in [0,1)^d$. Given a lattice point set $\Pcal_N=\{\bsx_0,\ldots,\bsx_{N-1}\}$ and drawing $\bsDelta$ from a uniform distribution over $[0,1)^d$, the corresponding \emph{randomly shifted lattice rule} applied to a function $g$ is then given by 
\[
   Q_{N,d,\bsDelta}(g):=\frac{1}{N}\sum_{k=0}^{N-1}g \left(\left\{\bsx_k + \bsDelta\right\}\right),
\]
i.e., all points of $\Pcal_N$ are shifted modulo one by the same $\bsDelta$.
If we would like to emphasize the role of the generating vector $\bsz$ of $\Pcal_N$, we also write $Q_{N,d,\bsz,\bsDelta}$ in the following. 

In the subsequent section, we will give our median algorithm based on lattice rules and analyze the corresponding approximation error.

\section{The median algorithm and its error}\label{sec:median}

As outlined in the introduction, we now turn to the definition of the \emph{median lattice-based algorithm}, which will be used for $L_2$-approximation in the weighted Korobov space $\Hcal_{d,\alpha,\bsgamma}$. Throughout this section, we assume downward-closed weights $\bsgamma$, that is, $0<\gamma_{\mathfrak{v}} \le \gamma_{\mathfrak{w}}\le 1$ whenever $\mathfrak{w}\subseteq \mathfrak{v}\subseteq \{1{:}d\}$, and a smoothness parameter $\alpha>1/2$. The parameters $\bsgamma$ and $\alpha$ are assumed to be fixed with these properties but otherwise arbitrary.

In this section, we first introduce the universal median lattice-based algorithm. We then state the main result of this paper, Theorem~\ref{thm:combined_error}, which provides a probabilistic error bound for the algorithm. The proof of this result relies on the error decomposition given in \eqref{eqn:errordecom}, while the detailed analysis of the individual terms in the decomposition is deferred to the next section. Finally, we discuss the tractability properties of the median lattice-based algorithm.

\subsection{The universal median lattice-based algorithm}
For a real number $L \geq 0$, we define
\[
\Acal_{d}(L) := \left\{ \bsh \in \ZZ^d : \prod_{j\in \supp(\bsh)} \abs{h_j} < L \right\},
\]
and
\[
\Acal_{d,\alpha,\bsgamma}(L) := \left\{ \bsh \in \ZZ^d : r_{2\alpha, \bsgamma}(\bsh) < L^{2\alpha} \right\}=\left\{ \bsh \in \ZZ^d : r_{1, \bsgamma^{1/(2\alpha)}}(\bsh) < L\right\},
\]
where $\bsgamma^{1/(2\alpha)} := (\gamma_{\fraku}^{1/(2\alpha)})_{\fraku \subseteq \{1{:}d\}}$.
Note that $\Acal_{d,\alpha,\bsgamma}(L) \subseteq \Acal_{d}(L)$, since $\sup_{\fraku \subseteq \{1{:}d\}} \gamma_{\fraku} \leq 1$ and $2\alpha > 1$.

For an odd positive integer $R$, we define the median of $R$ complex numbers $Z_1,\dots,Z_R$ by
\[
  \operatorname*{median}_{\substack{r\in\{1{:}R\}}} (Z_r)
  := \operatorname*{median}_{\substack{r\in\{1{:}R\}}} \Re(Z_r)
     + \icomp \,\cdot\, \operatorname*{median}_{\substack{r\in\{1{:}R\}}} \Im(Z_r),
\]
where $\Re(Z_r)$ and $\Im(Z_r)$ denote the real and imaginary parts of $Z_r$, respectively.  

In the following, we also assume that $N$ is an odd prime number. Although this requirement is not necessary for all results presented below, it is convenient to make this assumption from the outset, as some statements do rely on it.

\begin{algorithm}\label{alg:median}
    For given $R, d \in \NN$ and an odd prime $N$, perform the following steps:
    \begin{enumerate}
        \item \textbf{For} $r = 1, \dots, R$, do:
        \begin{enumerate}
            \item Randomly draw $\bsz_{r}$ from the uniform distribution over $\{1{:}(N-1)\}^d$.
            \item Randomly draw $\bsDelta_r$ from the uniform distribution over $[0,1)^d$.
            \item For $\bsh\in\ZZ^d$, define
        \[
        \widehat{f}_{N,\bsz_r,\bsDelta_r}(\bsh) :=
        \frac{1}{N} \sum_{k=0}^{N-1} 
        f\left( \left\{ \frac{k \bsz_r}{N}+\bsDelta_r \right\} \right)
        \exp\left(-2\pi \icomp \, \bsh \cdot \left( \frac{k \bsz_r}{N}+\bsDelta_r\right) \right).
        \]
        \end{enumerate}
        \item Order $\bsh \in \Acal_{d}(N/2)$ such that the medians of 
    $|\widehat{f}_{N,\bsz_r,\bsDelta_r}(\bsh)|^2$ over $r\in \{1{:}R\}$ 
    are in descending order, i.e.,
    \[
      \operatorname*{median}_{\substack{r\in\{1{:}R\}}} 
      |\widehat{f}_{N,\bsz_r,\bsDelta_r}(\bsh_1)|^2
      \ge 
      \operatorname*{median}_{\substack{r\in\{1{:}R\}}} 
      |\widehat{f}_{N,\bsz_r,\bsDelta_r}(\bsh_2)|^2
      \ge \cdots
    \]
    and set $K_N = \{\bsh_1, \dots, \bsh_N\}$.
    
       \item Define the median estimator
    \[
      \widehat{f}_{N,\bsz_{\{1{:}R\}},\bsDelta_{\{1{:}R\}}}(\bsh) :=
      \operatorname*{median}_{\substack{r\in\{1{:}R\}}} 
      \widehat{f}_{N,\bsz_r,\bsDelta_r}(\bsh),
    \]
    and the corresponding approximation
    \[
      (A^{\rand}_{N,\bsz_{\{1{:}R\}},\bsDelta_{\{1{:}R\}}}(f))(\bsx)
      := \sum_{\bsh \in K_N} 
      \widehat{f}_{N,\bsz_{\{1{:}R\}},\bsDelta_{\{1{:}R\}}}(\bsh)
      \exp (2 \pi \icomp \, \bsh \cdot \bsx),
    \]
    where $\bsz_{\{1{:}R\}}$ and $\bsDelta_{\{1{:}R\}}$ denote the collections 
    of $\bsz_r$ and $\bsDelta_r$, $r = 1, \dots, R$, respectively.
\end{enumerate}
\end{algorithm}

\begin{remark}
The key difference from the algorithm proposed in \cite{PKG24} is that the present method does not require prior specification of the smoothness parameter $\alpha$ or the weights $\bsgamma$. 
By considering the enlarged hyperbolic cross index set $\Acal_d(N/2)$ and sorting the estimated Fourier coefficients in descending order, the algorithm automatically adapts to the underlying smoothness and weights of the function.
\end{remark}

\begin{remark}\label{rem:evaluation_vs_work}
The total number of function evaluations is $M = R N$. 
Since we compute $\widehat{f}_{N,\bsz_r,\bsDelta_r}(\bsh)$ for all $\bsh \in \Acal_{d}(N/2)$, an additional computational effort of order $O(|\Acal_{d}(N/2)|)$ is required. 
It is known that $|\Acal_{d}(N/2)| \asymp N (\log N)^{d-1}$ (see, for instance, \cite[Chapter~2]{DTU18}), so that the associated computational cost can become significant for large $d$. 
Such an exponential dependence of the work on the dimension $d$ does not occur in the previous method studied in \cite{PKG24}.
This highlights a trade-off between achieving universality and controlling the computational budget.
\end{remark}

\subsection{Error analysis: the main result}

In this subsection, we first establish a few basic lemmas and then state the main result of this paper, Theorem~\ref{thm:combined_error}, which provides a probabilistic bound on the $L_2$-approximation error. This result is based on the error decomposition given in \eqref{eqn:errordecom}. The detailed derivation of the bounds for the individual terms in this decomposition is deferred to the next section.

We first estimate the cardinality of 
$\Acal_{d,\alpha,\bsgamma}(L)$. The following result was shown in \cite{CKNS20}.
\begin{lemma} \label{lem:calA_bound}
Let $\bsgamma=(\gamma_{\fraku})_{\fraku\subseteq \{1{:}d\}}$ 
be arbitrary positive weights. Then, for $\alpha>1/2$ and $\lambda > 1/(2\alpha)$,
\[
\abs{\Acal_{d,\alpha,\bsgamma}(L)} \leq 
T_{\alpha,\lambda}(\bsgamma)\, L^{2 \alpha\lambda} ,
\]
where 
\[
T_{\alpha,\lambda}(\bsgamma):=\sum_{\fraku\subseteq \{1:d\}} 
\gamma_{\fraku}^\lambda\, (2\zeta (2\alpha\lambda))^{\abs{\fraku}},
\]
and $\zeta(q):=\sum_{n=1}^\infty n^{-q}$ denotes the Riemann zeta function.
\end{lemma}

Next, for an odd prime $N$, we define
\begin{align}\label{eq:N_star_def} 
    N_* := \sup \left\{ L \ge 0 : |\Acal_{d,\alpha,\bsgamma}(L)| \le \frac{N-1}{16} \right\}.
\end{align}
Since $|\Acal_{d,\alpha,\bsgamma}(0)| = 0$, the above set is nonempty and thus $N_* > 0$. 
By definition, we then have
\begin{equation}\label{eqn:Nstar}
   |\Acal_{d,\alpha,\bsgamma}(N_*)| \le \frac{N-1}{16}.
\end{equation}
Note that $N_* < N/2$ because there are already $N$ integers between $-N/2$ and $N/2$. 
We also have the following lower bound on $N_*$:

\begin{lemma}\label{lem:Nstarlowerbound}
For an odd prime $N$, $\alpha>1/2$, $\lambda > 1/(2\alpha)$, and positive weights $\bsgamma=(\gamma_{\fraku})_{\fraku\subseteq \{1{:}d\}}$,
\[
N_* \ge \left( \frac{N-1}{16\, T_{\alpha,\lambda}(\bsgamma)} \right)^{1/(2\alpha\lambda)}.
\]
\end{lemma}

\begin{proof} 
Let $L > N_*$. By Lemma~\ref{lem:calA_bound} and the definition of $N_*$, we have
\[
\frac{N-1}{16} < |\Acal_{d,\alpha,\bsgamma}(L)| \le T_{\alpha,\lambda}(\bsgamma)\, L^{2 \alpha \lambda}.
\]
Taking the limit $L \downarrow N_*$ then yields the desired bound.
\end{proof}

Intuitively, $\Acal_{d,\alpha,\bsgamma}(N_*)$ contains all the large Fourier coefficients that we want our algorithm to estimate and include in $K_N$. 
The next lemma provides a bound on the error arising from leaving out frequencies that are not in $\Acal_{d,\alpha,\bsgamma}(N_*)$.

\begin{lemma}\label{lem:sumnotinL} 
For $L>0$ and $f \in \Hcal_{d,\alpha,\bsgamma}$,
    $$\sum_{\bsh \in \ZZ^d\setminus \Acal_{d,\alpha,\bsgamma}(L)}|\widehat{f}(\bsh)|^2\leq \frac{\|f\|^2_{d,\alpha,\bsgamma}}{L^{2\alpha} }.$$
\end{lemma}
\begin{proof}
We have
\[
\sum_{\bsh \in \ZZ^d \setminus \Acal_{d,\alpha,\bsgamma}(L)} |\widehat{f}(\bsh)|^2
\le \sum_{\bsh \in \ZZ^d \setminus \Acal_{d,\alpha,\bsgamma}(L)} 
|\widehat{f}(\bsh)|^2 \frac{r_{2\alpha, \bsgamma}(\bsh)}{L^{2\alpha}}
\le \frac{\|f\|^2_{d,\alpha,\bsgamma}}{L^{2\alpha}}. \qedhere
\]
\end{proof}

In the following, we write 
\[
N \cdot \ZZ^d := \{\bsl \in \ZZ^d : \ell_j \equiv 0 \tmod N \text{ for all } j \in \{1{:}d\}\}.
\]
The next result is an immediate corollary of Lemma~\ref{lem:sumnotinL}.
\begin{corollary}\label{cor:sumnotinN}
For an odd prime $N$ and $f \in \Hcal_{d,\alpha,\bsgamma}$,
\[
\sum_{\bsh \in \mathcal{A}_d(N/2)} \sum_{\bsl \in N \cdot \ZZ^d \setminus \{\bszero\}} 
|\widehat{f}(\bsh+\bsl)|^2
\le \frac{\|f\|^2_{d,\alpha,\bsgamma}}{N_*^{2\alpha}}.
\]
\end{corollary}
\begin{proof}
    Since $\bsh \in \mathcal{A}_d(N/2)$ implies $\bsh \in (-N/2,N/2)^d$ for an odd prime $N$, the shifted lattices $\bsh + N \cdot \ZZ^d$ are mutually disjoint for different $\bsh \in \mathcal{A}_d(N/2)$.  Moreover, for any $\bsh \in \mathcal{A}_d(N/2)$ and $\bsl \in N \cdot \ZZ^d \setminus \{\bszero\}$, we have $\bsh + \bsl \notin \mathcal{A}_d(N/2)$.  Hence, by Lemma~\ref{lem:sumnotinL},
    \[
\sum_{\bsh \in \mathcal{A}_d(N/2)} \sum_{\bsl \in N \cdot \ZZ^d \setminus \{\bszero\}} 
|\widehat{f}(\bsh+\bsl)|^2
\le \sum_{\bsh \in \ZZ^d \setminus \Acal_d(N/2)} |\widehat{f}(\bsh)|^2
\le \sum_{\bsh \in \ZZ^d \setminus \Acal_{d,\alpha,\bsgamma}(N/2)} |\widehat{f}(\bsh)|^2
\le \frac{\|f\|^2_{d,\alpha,\bsgamma}}{(N/2)^{2\alpha}}.
\]
The result then follows from the fact that $N_* < N/2$.
\end{proof}

To derive a probabilistic error bound for the algorithm $A^{\rand}_{N,\bsz_{\{1{:}R\}},\bsDelta_{\{1{:}R\}}}$, we introduce some 
additional notation and an auxiliary set of indices. For the set $K_N$ generated by our algorithm, define
\[
f_{K_N} := \sum_{\bsh \in K_N} \widehat{f}(\bsh) \exp(2 \pi \icomp \bsh \cdot \bsx).
\]
Moreover, consider an alternative approximation to $f$ using an oracle with access 
to all Fourier coefficients of $f$. For $\bsh' \in \mathcal{A}_d(N/2)$, define
\begin{align}\label{eq:def_F_bsh'}
F(\bsh') := \sum_{\bsl \in N \cdot \ZZ^d} |\widehat{f}(\bsh' + \bsl)|^2.
\end{align}
Ordering the elements $\bsh'$ of $\mathcal{A}_d(N/2)$ so that 
$F(\bsh'_1) \ge F(\bsh'_2) \ge \cdots$, we set 
$K'_N = \{\bsh'_1, \dots, \bsh'_N\}$.

The main idea of the remaining proof is to use the triangle inequality to obtain an upper bound on the squared error:
\begin{align}
& \|f - A^{\rand}_{N,\bsz_{\{1{:}R\}},\bsDelta_{\{1{:}R\}}}(f)\|^2_{L_2} \notag \\
&\le \Big(\|f_{K_N} - A^{\rand}_{N,\bsz_{\{1{:}R\}},\bsDelta_{\{1{:}R\}}}(f)\|_{L_2} 
+ \|f - f_{K_N}\|_{L_2}\Big)^2 \notag \\
&\le 2 \|f_{K_N} - A^{\rand}_{N,\bsz_{\{1{:}R\}},\bsDelta_{\{1{:}R\}}}(f)\|^2_{L_2} 
+ 2 \|f - f_{K_N}\|^2_{L_2} \notag \\
&= 2 \sum_{\bsh \in K_N} 
\left| \widehat{f}_{N,\bsz_{\{1{:}R\}},\bsDelta_{\{1{:}R\}}}(\bsh) - \widehat{f}(\bsh) \right|^2
+ 2 \sum_{\bsh \in \ZZ^d \setminus K_N} |\widehat{f}(\bsh)|^2 \notag \\
&\le 2 \sum_{\bsh \in K_N} 
\left| \widehat{f}_{N,\bsz_{\{1{:}R\}},\bsDelta_{\{1{:}R\}}}(\bsh) - \widehat{f}(\bsh) \right|^2
+ 2 \sum_{\bsh \in \ZZ^d \setminus K'_N} |\widehat{f}(\bsh)|^2
+ 2 \sum_{\bsh \in K'_N \setminus K_N} |\widehat{f}(\bsh)|^2, \tag{$\clubsuit$} \label{eqn:errordecom}
\end{align}
where we have used the orthonormality of the Fourier basis in the equality.

We will analyze the three summands in the last line separately in the next section.
The corresponding probabilistic bounds are established in Theorems~\ref{thm:L2rate_constants}, \ref{lem:fminusfK'}, and \ref{thm:K'minusK}, respectively.
By combining these results, we arrive at the following main theorem.

\begin{theorem}\label{thm:combined_error}
Let $f\in \Hcal_{d,\alpha,\bsgamma}$, and let $N$ be an odd prime and $R$ an odd natural number, chosen sufficiently large such that $\varepsilon_1+\varepsilon_2\leq \varepsilon$ for some $\varepsilon\in (0,1)$, where $\varepsilon_1$ and $\varepsilon_2$ are defined in \eqref{eq:choose_R_first} and \eqref{eq:choose_R_second}, respectively. Then, with probability at least $1-\varepsilon$, we have
\[
  \|f-A^{\rand}_{N,\bsz_{\{1{:}R\}},\bsDelta_{\{1{:}R\}}}(f)\|^2_{L_2}
  \le \frac{32\|f\|^2_{d,\alpha,\bsgamma}(2N-1)}{N^{2\alpha}_*(N-1)}
  \left(\frac{991N-255}{120(2N-1)}
  + \max\!\left(960,\sqrt{240 N^{4\alpha}_*\Bar{r}_{\mathrm{sum}}}\right)\right),
\]
where $N_*$ is as defined in \eqref{eq:N_star_def} and
\begin{align}\label{eq:def_r_sum}
\Bar{r}_{\mathrm{sum}}=16\left(\frac{4\alpha}{N-1}\int_{N_*}^\infty \frac{ |\Acal_{d,\alpha,\bsgamma}(L)|}{L^{4\alpha+1} }\rd L
  +  \frac{1}{N^{4\alpha}} 
    \sum_{\emptyset \neq \fraku \subseteq \{1{:}d\}} \gamma^{2}_{\fraku} \left(2^{1+4\alpha}\zeta(4\alpha)\right)^{\abs{\fraku}}\right).
\end{align}
\end{theorem}

With the $(\epsilon,\delta)$-approximation framework \eqref{eq:eps-delta_framework} in mind, our randomized algorithm achieves an $(\epsilon,\delta)$-approximation with
\[
\epsilon^2 = \frac{32(2N-1)}{N^{2\alpha}_*(N-1)}
  \left(\frac{991N-255}{120(2N-1)}
  + \max\!\left(960,\sqrt{240 N^{4\alpha}_*\Bar{r}_{\mathrm{sum}}}\right)\right),
\]
for $\delta \le \varepsilon_1+\varepsilon_2 \in (0,1)$.

\begin{remark}

Given a failure probability tolerance $\varepsilon \in (0,1)$, the condition $\varepsilon_1+\varepsilon_2 \leq \varepsilon$ holds, where $\varepsilon_1$ and $\varepsilon_2$ are given in \eqref{eq:choose_R_first} and \eqref{eq:choose_R_second}, respectively, provided that
\[
N\left(\frac{1}{2}\right)^{R/2} 
+ \frac{|\mathcal{A}_d(N/2)|}{2}\left(\frac{1}{2}\right)^{R/2} 
+ \frac{N}{2}\left(\frac{3}{4}\right)^{R/2} \leq \varepsilon.
\]
This inequality is satisfied if
\[
N\left(\frac{3}{4}\right)^{R/2} \leq \frac{\varepsilon}{3} 
\quad \text{and} \quad  
|\mathcal{A}_d(N/2)|\left(\frac{1}{2}\right)^{R/2} \leq \varepsilon.
\]

The first condition holds whenever $R \geq 2\log (3N/\varepsilon)/(\log (4/3))$.
Next, since $\Acal_{d}(N/2)$ is a special instance of $\Acal_{d,\alpha,\bsgamma}(L)$, Lemma~\ref{lem:calA_bound} with $L=N/2$, $\lambda=1/\alpha$, and $\gamma_u=1$ for all $u\subseteq \{1{:}d\}$ implies
\[
|\Acal_{d}(N/2)| \leq \frac{N^2}{4}\,(1+2\zeta(2))^{d}.
\]
Thus, the second condition holds if
\[
R \;\geq\; 2\left( \frac{\log (3N^2)-\log (4\varepsilon)}{\log (2)} + \frac{\log(1+2\zeta(2))}{\log (2)}\,d\right).
\]
We conclude that the error bound in Theorem~\ref{thm:combined_error} fails with probability at most $\varepsilon$ if
\begin{equation}\label{eq:R_large_enough}
    R \geq 2\max\left( \frac{\log (3N/\varepsilon)}{\log (4/3)}, \; \frac{\log (3N^2)-\log (4\varepsilon)}{\log (2)} + \frac{\log (1+2\zeta(2))}{\log (2)}\,d \right).
\end{equation}
\end{remark}

\subsection{Tractability analysis}\label{subsec:tractability}

The probabilistic bound for the $L_2$-approximation error established in Theorem~\ref{thm:combined_error} still depends on the quantities $N_*$ and $\Bar{r}_{\mathrm{sum}}$, and thus does not yet yield an explicit error bound that is independent of these. In this subsection, we analyze these quantities under suitable assumptions and derive bounds that either depend only polynomially on the dimension $d$ or are independent of $d$. This, in turn, allows us to investigate the tractability properties of the universal median lattice-based algorithm.

First, let us consider the quantity $\Bar{r}_{\mathrm{sum}}$ defined in \eqref{eq:def_r_sum}.
By Lemma~\ref{lem:calA_bound}, for any $\lambda \in (1/(2\alpha),2)$ we obtain
\[ 
\int_{N^\ast}^\infty \frac{\abs{\Acal_{d,\alpha,\bsgamma}(L)}}{L^{4\alpha +1}}\rd L
\le  T_{\alpha,\lambda}(\bsgamma) \int_{N^\ast}^\infty \frac{1}{L^{4\alpha +1-2\alpha\lambda}}\rd L 
= \frac{T_{\alpha,\lambda}(\bsgamma)}{2\alpha(2-\lambda)}\, \frac{1}{N_*^{4\alpha-2\alpha\lambda}}. 
\]
This implies that
\[
\Bar{r}_{\mathrm{sum}} \le 
\frac{32}{2-\lambda}\,\frac{T_{\alpha,\lambda}(\bsgamma)}{(N-1)N_*^{4\alpha-2\alpha\lambda}}
  +  \frac{16}{N^{4\alpha}} 
    \sum_{\emptyset \neq \fraku \subseteq \{1:d\}} \gamma^2_{\fraku}
\left(2^{1+4\alpha}\zeta(4\alpha)\right)^{\abs{\fraku}}.
\]
Moreover, since $N_* < N/2$ always holds, we further obtain
\[
N^{4\alpha}_*\Bar{r}_{\mathrm{sum}} 
\le 
\frac{32}{2-\lambda}\, \frac{N^{2\alpha\lambda} T_{\alpha,\lambda} (\bsgamma)}{2^{2\alpha\lambda}(N-1)}
+  \frac{16}{2^{4\alpha} } 
    \sum_{\emptyset \neq \fraku \subseteq \{1:d\}} \gamma^2_{\fraku}
\left(2^{1+4\alpha}\zeta(4\alpha)\right)^{\abs{\fraku}}.
\]

Combining the above estimate with Lemma~\ref{lem:Nstarlowerbound} and Theorem~\ref{thm:combined_error}, we obtain the following result.

\begin{corollary}\label{cor:L_2_bound}
For any $f \in \Hcal_{d,\alpha,\bsgamma}$ and $\varepsilon\in (0,1)$, if the odd prime $N$ and the odd natural number $R$ satisfy condition~\eqref{eq:R_large_enough}, then with probability at least $1-\varepsilon$, we have
\begin{align}
    & \|f-A^{\rand}_{N,\bsz_{\{1{:}R\}},\bsDelta_{\{1{:}R\}}}(f)\|^2_{L_2} \notag \\
    & \leq 
    \norm{f}^2_{d,\alpha,\bsgamma}
    \frac{C_1}{N^{2\alpha}_*}\left(1 + C_{2}\sqrt {N^{4\alpha}_*\Bar{r}_{\mathrm{sum}}}\right) \notag \\
    & \leq 
    \norm{f}^2_{d,\alpha,\bsgamma}
    \frac{C_{1,\alpha,\lambda}(T_{\alpha,\lambda}(\bsgamma))^{1/\lambda}}{N^{1/\lambda}}\left(1 + C_{2,\alpha,\lambda}\sqrt{N^{2\alpha\lambda-1}T_{\alpha,\lambda} (\bsgamma)+ \sum_{\emptyset \neq \fraku \subseteq \{1:d\}} \gamma^2_{\fraku}
\left(2^{1+4\alpha}\zeta(4\alpha)\right)^{\abs{\fraku}}}\right),\label{eq:overall_bound}
\end{align}
for any $\lambda \in (1/(2\alpha),2)$, where $C_1, C_2, C_{1,\alpha,\lambda}, C_{2,\alpha,\lambda}$ are constants independent of $N$, $R$, and $d$, but possibly depending on $\alpha$ and $\lambda$, as indicated by their subscripts.
\end{corollary}

From Corollary~\ref{cor:L_2_bound}, we can draw several important conclusions regarding the median lattice-based algorithm. First, by choosing $\lambda$ arbitrarily close to $1/(2\alpha)$, we can achieve, for any given $f$, a convergence rate for the $L_2$-approximation error that is arbitrarily close to the optimal rate $N^{-\alpha}$, with high probability. Since our algorithm does not require any prior knowledge of the function’s smoothness or the weight parameters, this near-optimal convergence rate applies universally for any smoothness $\alpha > 1/2$ and any downward-closed weights $\bsgamma$.

Next, for $q\ge 0$, define
\[
B_{\lambda,q}:= \sup_{d\in\NN} 
\left(\frac{1}{d^q}\, T_{\alpha,\lambda}(\bsgamma)\right)
\quad
\text{and}\quad 
\widetilde{B}_{q}:= \sup_{d\in\NN} 
\left(\frac{1}{d^q}\, \sum_{\fraku \subseteq \{1:d\}} \gamma^2_{\fraku}
\left(2^{1+4\alpha}\zeta(4\alpha)\right)^{\abs{\fraku}}\right).
\]
If, for some $\lambda\in (1/(2\alpha),2)$ and some $q>0$, we have $B_{\lambda,q}<\infty$ and $\widetilde{B}_q<\infty$, then the error bound in \eqref{eq:overall_bound} depends at most polynomially on $d$, and an error of at most $\varepsilon>0$ can be achieved with $N$ depending at most polynomially on $d$ and $\varepsilon$, with high probability. Moreover, the number $R$ of independent rank-1 lattices required for this result depends only logarithmically on $N$ and linearly on $d$ by \eqref{eq:R_large_enough}. This situation is referred to as \textit{polynomial tractability} in the context of \textit{Information-Based Complexity} (see, e.g., \cite{novak2008tractability}).

In the case of \emph{product weights}, where $\gamma_{\fraku}=\prod_{j\in \fraku}\gamma_j$ with $\gamma_j\leq 1$, we have
\[ T_{\alpha,\lambda}(\bsgamma)=\sum_{\fraku\subseteq \{1:d\}} 
\prod_{j\in u}\left( 2\gamma_j^{\lambda} \zeta(2\alpha\lambda)\right)=\prod_{j=1}^{d}\left( 1+2\gamma_j^{\lambda} \zeta(2\alpha\lambda)\right). \]
Assume that 
\begin{align}\label{eq:condition_poly_tractable}
    A_{\lambda} := \limsup_{d\to \infty}\frac{1}{\log d}\sum_{j=1}^{d}\gamma_j^{\lambda}<\infty,
\end{align}
for some $\lambda\in (1/(2\alpha),2)$.
This assumption implies that, for any fixed $\delta>0$, there exists $d_\delta$ such that
\[ \frac{1}{\log d}\sum_{j=1}^{d}\gamma_j^{\lambda}<A_{\lambda}+\delta,\qquad \text{for all $d\ge d_{\delta}$.}\]
Applying the elementary inequality $\exp(x)\ge x+1$, we then obtain
\begin{align*}
    \frac{1}{d^q}\prod_{j=1}^{d}\left( 1+2\gamma_j^{\lambda} \zeta(2\alpha\lambda)\right) & \le \frac{1}{d^q}\exp\left(2\zeta(2\alpha\lambda)\sum_{j=1}^{d}\gamma_j^{\lambda}\right)  \\
    & \le \frac{1}{d^q}\exp\left(2(A_{\lambda}+\delta)\zeta(2\alpha\lambda)\log d \right) =d^{2(A_{\lambda}+\delta)\zeta(2\alpha\lambda)-q}.
\end{align*}
Hence, $B_{\lambda,q}<\infty$ for any $q> 2 A_{\lambda}\zeta(2\alpha\lambda)$. Similarly, if \eqref{eq:condition_poly_tractable} holds with $\lambda=2$, then $\widetilde{B}_{q}<\infty$ for any $q>2A_2\zeta(4\alpha)$.

If $B_{\lambda,q}<\infty$ and $\widetilde{B}_q<\infty$ even hold for $q=0$, then the error bound becomes independent of $d$. In the case of product weights, these conditions simplify to
\[ \sum_{j=1}^{\infty}\gamma_j^{\lambda}<\infty\quad \text{and}\quad \sum_{j=1}^{\infty}\gamma_j^2<\infty, \]
respectively.
Unfortunately, the required number $R$ still depends linearly on $d$ (cf.\ \eqref{eq:R_large_enough}), which prevents us from claiming \textit{strong polynomial tractability} for our approximation algorithm. This limitation appears to be a consequence of the universality of the algorithm. Nevertheless, the linear dependence on $d$ may be acceptable for applications with moderate dimension. If strong tractability is desired, one could instead use the algorithm from \cite{PKG24}, which achieves the same near-optimal convergence order but does not maintain universality.

\section{Error analysis: bounds on the error decomposition}\label{sec:error}

In this section, we provide a detailed analysis of the three summands in the error decomposition \eqref{eqn:errordecom}, which underlies the probabilistic $L_2$-error bounds given in Theorem~\ref{thm:combined_error} and Corollary~\ref{cor:L_2_bound}. Each term is treated in a separate subsection, where we establish probabilistic bounds that, when combined, recover the main result presented in Section~\ref{sec:median}.

\subsection{The first term of \eqref{eqn:errordecom}}

Our proof in this part is adapted from that in \cite{PKG24}. 
First, we show that, for most random choices of $\bsz_r$, the generated lattices are ``good'' for estimating $\widehat{f}(\bsh)$.

\begin{lemma}\label{lem:nointersect}
For any $r\in\{1{:}R\}$, define
\begin{align}\label{eq:def_dual_lattice}
P_{N,\bsz_r}^\perp := \{\bsl \in \ZZ^d : \bsz_r^\top \bsl \equiv 0 \tmod N\}.
\end{align}
For any $\bsh \in \Acal_d(N/2)$, define
\[
\Bcal_{\bsh} := \left\{ \bsl \in \ZZ^d \setminus (N \cdot \ZZ^d) : r_{2\alpha,\bsgamma}(\bsh+\bsl) < N_*^{2\alpha} \right\},
\]
where $N_*$ is as defined in \eqref{eq:N_star_def}. Then
\[
\Pr(P_{N,\bsz_r}^\perp \cap \Bcal_{\bsh} \neq \emptyset) \le \frac{1}{16}.
\]
\end{lemma}
\begin{proof} 
By \eqref{eqn:Nstar}, we have
\begin{align*}
|\Bcal_{\bsh}| 
&= \Big| \big\{ \bsl \in \ZZ^d \setminus (N\cdot \ZZ^d) : r_{2\alpha,\bsgamma}(\bsh+\bsl) < N_*^{2\alpha} \big\} \Big| \\
&\le \Big| \big\{ \bsl \in \ZZ^d : r_{2\alpha,\bsgamma}(\bsl) < N_*^{2\alpha} \big\} \Big| \le \frac{N-1}{16}.
\end{align*}
For each $\bsl \in \ZZ^d \setminus (N \cdot \ZZ^d)$, 
\begin{align}\label{eqn:ellinP}
\Pr(\bsz_r^\top \bsl \equiv 0 \tmod N) = \frac{1}{(N-1)^d} \sum_{\bsz_r \in \{1{:}N-1\}^d} \mathbf{1}\{\bsl \in P_{N,\bsz_r}^\perp\} \le \frac{1}{N-1}.
\end{align}
It then follows from a union bound that
\[
\Pr(P_{N,\bsz_r}^\perp \cap \Bcal_{\bsh} \neq \emptyset) 
\le \frac{|\Bcal_{\bsh}|}{N-1} \le \frac{1}{16}. \qedhere
\]
\end{proof}

The next lemma bounds the typical error in replacing $\widehat{f}(\bsh)$ with $\widehat{f}_{N,\bsz_r,\bsDelta_r}(\bsh)$.

\begin{lemma}\label{lem:concentration}
Let $f \in \Hcal_{d,\alpha,\bsgamma}$. For 
$\bsh \in \Acal_d(N/2)$, define
\[
\delta(\bsh) := 4 \left( \frac{\|f\|^2_{d,\alpha,\bsgamma}}{N_*^{2\alpha}(N-1)} + \sum_{\bsl \in N \cdot \ZZ^d \setminus \{\bszero\}} |\widehat{f}(\bsh+\bsl)|^2 \right)^{1/2}.
\]
Then, for any $r \in \{1{:}R\}$,
\[
\Pr\left( |\widehat{f}_{N,\bsz_r,\bsDelta_r}(\bsh) - \widehat{f}(\bsh)|^2 > \delta(\bsh)^2 \right) \le \frac{1}{8}.
\]
\end{lemma}

\begin{proof}
By Lemma~\ref{lem:nointersect} and Markov's inequality, it holds that
\begin{align}\label{eqn:Markov}
    & \Pr\left(|\widehat{f}_{N,\bsz_r,\bsDelta_r}(\bsh)-\widehat{f}(\bsh)|^2>\delta(\bsh)^2\right)\nonumber\\
    & \le \Pr(P_{N,\boldsymbol{z}_r}^\perp\cap \mathcal{B}_{\bsh}\neq \emptyset)+\Pr\left(\{|\widehat{f}_{N,\bsz_r,\bsDelta_r}(\bsh)-\widehat{f}(\bsh)|^2> \delta(\bsh)^2\}\cap \{P_{N,\boldsymbol{z}_r}^\perp\cap \mathcal{B}_{\bsh}= \emptyset\}\right)\nonumber\\
    & = \Pr(P_{N,\boldsymbol{z}_r}^\perp\cap \mathcal{B}_{\bsh}\neq \emptyset)+\Pr\left(\delta(\bsh)^{-2}|\widehat{f}_{N,\bsz_r,\bsDelta_r}(\bsh)-\widehat{f}(\bsh)|^2\bsone\{P_{N,\boldsymbol{z}_r}^\perp\cap \mathcal{B}_{\bsh}= \emptyset\}> 1\right) \nonumber\\
    & \le \frac{1}{16} 
    +\frac{1}{\delta(\bsh)^2}\e\left[|\widehat{f}_{N,\bsz_r,\bsDelta_r}(\bsh)-\widehat{f}(\bsh)|^2\bsone\{P_{N,\boldsymbol{z}_r}^\perp\cap \mathcal{B}_{\bsh}= \emptyset\}\right].
\end{align}
As outlined in \cite[Proof of Lemma 4.1]{CGK24},
\begin{equation}\label{eqn:deltamean}
\int_{[0,1)^d} |\widehat{f}_{N,\bsz_r,\bsDelta_r}(\bsh)-\widehat{f}(\bsh)|^2 \, d\bsDelta_r
= \sum_{\bsl \in P_{N,\bsz_r}^\perp \setminus \{\bszero\}} |\widehat{f}(\bsh+\bsl)|^2.
\end{equation}
Taking expectation over $\bsz_r$ and using independence of $\bsDelta_r$ and $\bsz_r$, we obtain
\begin{align*}
    & \e\left[|\widehat{f}_{N,\bsz_r,\bsDelta_r}(\bsh)-\widehat{f}(\bsh)|^2\bsone\{P_{N,\boldsymbol{z}_r}^\perp\cap \mathcal{B}_{\bsh}= \emptyset\}\right] \notag \\
    & = \frac{1}{(N-1)^d}\sum_{\bsz_r\in \{1{:}(N-1)\}^d}\sum_{\bsl \in P_{N,\boldsymbol{z}_r}^\perp\setminus \{\bszero\}}|\widehat{f}(\bsh+\bsl)|^2 \bsone\{P_{N,\boldsymbol{z}_r}^\perp\cap \mathcal{B}_{\bsh}= \emptyset\}\notag\\
    & \le \frac{1}{(N-1)^d}\sum_{\bsz_r\in \{1{:}(N-1)\}^d}\left(\sum_{\bsl \in P_{N,\boldsymbol{z}_r}^\perp\setminus (N\cdot \ZZ^d) }|\widehat{f}(\bsh+\bsl)|^2 \frac{r_{2\alpha, \bsgamma}(\bsh+\bsl) }{N_*^{2\alpha}}+\sum_{\bsl \in N\cdot \ZZ^d\setminus \{\bszero\} }|\widehat{f}(\bsh+\bsl)|^2\right) \notag\\
    & = \sum_{\bsl \in \ZZ^d\setminus (N\cdot \ZZ^d)}|\widehat{f}(\bsh+\bsl)|^2\frac{r_{2\alpha, \bsgamma}(\bsh+\bsl) }{N_*^{2\alpha}}\left(\frac{1}{(N-1)^d}\sum_{\bsz_r\in \{1{:}(N-1)\}^d}\bsone\{\bsl \in P_{N,\boldsymbol{z}_r}^\perp\}\right)+\sum_{\bsl \in N\cdot \ZZ^d\setminus \{\bszero\} }|\widehat{f}(\bsh+\bsl)|^2\notag\\
    & \le \frac{\|f\|^2_{d,\alpha,\bsgamma}}{N_*^{2\alpha}(N-1) }+\sum_{\bsl \in N\cdot \ZZ^d\setminus \{\bszero\} }|\widehat{f}(\bsh+\bsl)|^2,
\end{align*}
where we have used \eqref{eqn:ellinP}. Substituting this bound into \eqref{eqn:Markov} yields the desired probability bound.
\end{proof}

The next lemma illustrates how the median trick reduces the failure probability. It comes from \cite[Proposition 2.2]{kunsch2019solvable}, which is a minor modification of  \cite[Proposition 2.1]{niemiro2009fixed}.

\begin{lemma}\label{lem:mediantrick}
    Let $R$ be an odd natural number, and let $X_1, X_2, \dots, X_R$ be $R$ independent and identically distributed real random variables. Suppose that for some interval $I \subseteq \RR$ and $p \in (0,1/2)$ we have
    \[
    \Pr\left(X_1 \notin I\right) \leq p.
    \]
    Then
    \[
    \Pr\left(\operatorname*{median}_{r \in \{1{:}R\}}(X_r)     \notin I\right) 
    \leq \frac{1}{2}\,(4p(1-p))^{R/2}.
    \]
\end{lemma}

The following result is a direct consequence of Lemmas~\ref{lem:concentration} and~\ref{lem:mediantrick}.

\begin{lemma}\label{lem:medianbound}
    Suppose that $R$ is an odd integer and $\bsh \in \Acal_{d}(N/2)$. Then
    \[
    \Pr\left(
    \left|\widehat{f}_{N,\bsz_{\{1{:}R\}},\bsDelta_{\{1{:}R\}}}(\bsh) - \widehat{f}(\bsh)\right|^2 > 2\,\delta(\bsh)^2 \right) \le 2^{-R/2},
\]
where $\delta(\bsh)$ is as defined in Lemma~\ref{lem:concentration}.
\end{lemma}
\begin{proof} 
    By Lemma~\ref{lem:concentration}, for each $r \in \{1{:}R\}$ we have
    \[
    \Pr\left(|\Re\widehat{f}_{N,\bsz_r,\bsDelta_r}(\bsh)-\Re\widehat{f}(\bsh)|^2> \delta(\bsh)^2\right)
    \le \frac{1}{8}. 
    \]
    Applying Lemma~\ref{lem:mediantrick}, it follows that
    \[
        \Pr\left(\left|\Re\widehat{f}_{N,\bsz_{\{1{:}R\}},\bsDelta_{\{1{:}R\}}}(\bsh)-\Re\widehat{f}(\bsh)\right|^2> \delta(\bsh)^2\right)
         \le \frac{1}{2}\left( 4\cdot \frac{1}{8}\cdot\frac{7}{8}\right)^{R/2}\leq \frac{1}{2}2^{-R/2}.
    \]
    A similar bound holds for the imaginary part. The claim then follows by applying a union bound.    
\end{proof}

We now show a bound on the first summand of \eqref{eqn:errordecom}.

\begin{theorem}\label{thm:L2rate_constants}
    Let $N$ be an odd prime and let $R$ be an odd natural number satisfying
    \begin{equation}\label{eq:choose_R_first}
    \varepsilon_1:= N\,2^{-R/2}< 1.
    \end{equation}
    Then, for any $f\in \Hcal_{d,\alpha,\bsgamma}$, with probability at least $1-\varepsilon_1$ we have
    \[
    \sum_{\bsh\in K_N}\left|\widehat{f}_{N,\bsz_{\{1{:}R\}},\bsDelta_{\{1{:}R\}}}(\bsh)-\widehat{f}(\bsh)\right|^2\le 32\frac{\|f\|^2_{d,\alpha,\bsgamma}}{N_*^{2\alpha}}\;\frac{2N-1}{N-1}.
    \]
\end{theorem}

\begin{proof}
    By Lemma~\ref{lem:medianbound}, 
    \begin{align*}
    \Pr\left(\left|\widehat{f}_{N,\bsz_{\{1{:}R\}},\bsDelta_{\{1{:}R\}}}(\bsh)-\widehat{f}(\bsh)\right|^2> 2\delta(\bsh)^2 \text{ for at least one }  \bsh\in K_N\right)
    \le |K_N| \, \,  2^{-R/2}
    = \varepsilon_1.
    \end{align*}
    Hence, with probability at least $1-\varepsilon_1$, we have $\left|\widehat{f}_{N,\bsz_{\{1{:}R\}},\bsDelta_{\{1{:}R\}}}(\bsh)-\widehat{f}(\bsh)\right|^2\leq 2\delta(\bsh)^2$ for all $\bsh\in K_N$. It then follows that
    \begin{align*}
        \sum_{\bsh\in K_N}\left|\widehat{f}_{N,\bsz_{\{1{:}R\}},\bsDelta_{\{1{:}R\}}}(\bsh)-\widehat{f}(\bsh)\right|^2
        & \le 2\sum_{\bsh\in K_N}\delta(\bsh)^2\\
        & \le 32\left(\frac{|K_N|}{N-1}\,\frac{\|f\|^2_{d,\alpha,\bsgamma}}{N_*^{2\alpha}}+\sum_{\bsh\in K_N}\sum_{\bsl \in N\cdot \ZZ^d\setminus \{\bszero\} }|\widehat{f}(\bsh+\bsl)|^2\right)\\
        & \le 32\left(\frac{N}{N-1}\,\frac{\|f\|^2_{d,\alpha,\bsgamma}}{N_*^{2\alpha}}+\frac{\|f\|^2_{d,\alpha,\bsgamma}}{N_*^{2\alpha}}\right),
    \end{align*}
    where the last inequality uses Corollary~\ref{cor:sumnotinN}.
\end{proof}

\subsection{The second term of \eqref{eqn:errordecom}}

Next, we turn to the second term in the error decomposition \eqref{eqn:errordecom}. Recall that $F(\bsh')$ is defined in \eqref{eq:def_F_bsh'} and that the set $K'_N=\{\bsh'_1,\dots,\bsh'_N\}\subset \Acal_d(N/2)$ consists of the indices corresponding to the $N$ largest values of $F(\bsh')$. Consequently, the contribution of the Fourier coefficients outside $K'_N$ can be suitably bounded, as proven in the following theorem.

\begin{theorem} \label{lem:fminusfK'}
For any $f\in \Hcal_{d,\alpha,\bsgamma}$, we have
\[
\sum_{\bsh\in \ZZ^d\setminus K'_N}  |\widehat{f}(\bsh)|^2\leq \frac{2\|f\|^2_{d,\alpha,\bsgamma}}{N^{2\alpha}_*},\]
where $N_*$ is as defined in \eqref{eq:N_star_def}.
\end{theorem}
\begin{proof}
Since $\Acal_{d,\alpha,\bsgamma}(N_*)\subseteq \mathcal{A}_d(N/2)$ and $|\Acal_{d,\alpha,\bsgamma}(N_*)| \leq N$, it follows that
\begin{align*}
    \sum_{\bsh \in \ZZ^d\setminus \Acal_{d,\alpha,\bsgamma}(N_*)}|\widehat{f}(\bsh)|^2
    =& \sum_{\bsh\in \ZZ^d} |\widehat{f}(\bsh)|^2 - \sum_{\bsh \in \Acal_{d,\alpha,\bsgamma}(N_*)}|\widehat{f}(\bsh)|^2 \\
    \geq & \sum_{\bsh\in \ZZ^d} |\widehat{f}(\bsh)|^2 - \sum_{\bsh \in \Acal_{d,\alpha,\bsgamma}(N_*)}F(\bsh)\\
    \geq & \sum_{\bsh\in \ZZ^d} |\widehat{f}(\bsh)|^2 - \sum_{\bsh' \in K'_N}F(\bsh')\\
    = & \sum_{\bsh\in \ZZ^d} |\widehat{f}(\bsh)|^2 -\sum_{\bsh'\in K'_N} \sum_{\bsl\in N\cdot \ZZ^d} |\widehat{f}(\bsh'+\bsl)|^2 \\
    = & \sum_{\bsh\in \ZZ^d\setminus K'_N} |\widehat{f}(\bsh)|^2-\sum_{\bsh'\in K'_N} \sum_{\bsl\in N\cdot \ZZ^d\setminus \{\bszero\}} |\widehat{f}(\bsh'+\bsl)|^2.
\end{align*}
Therefore,
\[
\sum_{\bsh\in \ZZ^d\setminus K'_N} |\widehat{f}(\bsh)|^2 \leq \sum_{\bsh \in \ZZ^d\setminus \Acal_{d,\alpha,\bsgamma}(N_*)}|\widehat{f}(\bsh)|^2+\sum_{\bsh'\in K'_N} \sum_{\bsl\in N\cdot \ZZ^d\setminus \{\bszero\}} |\widehat{f}(\bsh'+\bsl)|^2.
\]
The desired bound follows by applying Lemma~\ref{lem:sumnotinL} and Corollary~\ref{cor:sumnotinN}.
\end{proof}

\subsection{The third term of \eqref{eqn:errordecom}}

Finally, we consider the third term in the error decomposition \eqref{eqn:errordecom}.
In contrast to the first two terms, bounding this term requires a more detailed analysis.
The key idea is to show that, with high probability, all $\bsh' \in K'_N$ corresponding to large values of $|\widehat{f}(\bsh')|^2$ are contained in $K_N$.

To this end, we first establish several lemmas to study the typical size of $|\widehat{f}_{N,\bsz_r,\bsDelta_r}(\bsh)|^2$, 
focusing in particular on its variance with respect to the random shift $\bsDelta_r$ conditioned on a fixed $\bsz_r$.

\begin{lemma} 
For any $r\in\{1{:}R\}$ and $\bsh\in \mathcal{A}_d(N/2)$, we have
\begin{align}
    \int_{[0,1)^d} \bigl|\widehat{f}_{N,\bsz_r,\bsDelta_r}(\bsh)\bigr|^2 \,\mathrm{d}\bsDelta_r
    &= \sum_{\bsl \in P_{N,\boldsymbol{z}_r}^\perp} |\widehat{f}(\bsh+\bsl)|^2, \label{eqn:1stmoment} \\
    \int_{[0,1)^d} \bigl|\widehat{f}_{N,\bsz_r,\bsDelta_r}(\bsh)\bigr|^4 \,\mathrm{d}\bsDelta_r
    &= \sum_{\bsl \in P_{N,\boldsymbol{z}_r}^\perp} \left| \sum_{\bsl' \in P_{N,\boldsymbol{z}_r}^\perp} \widehat{f}(\bsh+\bsl') \, \widehat{f}(\bsh+\bsl-\bsl') \right|^2, \label{eqn:2ndmoment}
\end{align}
where $P_{N,\boldsymbol{z}_r}^\perp$ denotes the dual lattice of $P_{N,\boldsymbol{z}_r}$, as defined in \eqref{eq:def_dual_lattice}.
\end{lemma}

\begin{proof}
The first equality follows similarly to \eqref{eqn:deltamean}. 
For the second equality, we start from \cite[Proof of Lemma 4.1]{CGK24}:
\[
\widehat{f}_{N,\bsz_r,\bsDelta_r}(\bsh) = \sum_{\bsl \in P_{N,\boldsymbol{z}_r}^\perp} \widehat{f}(\bsh+\bsl)\, \exp(2\pi\icomp \bsl\cdot\bsDelta_r).
\]
It follows that
\[ 
    |\widehat{f}_{N,\bsz_r,\bsDelta_r}(\bsh)|^4=\abs{\sum_{\bsl,\bsl' \in P_{N,\boldsymbol{z}_r}^\perp}\widehat{f}(\bsh+\bsl) \widehat{f}(\bsh+\bsl')\exp(2\pi\icomp (\bsl+\bsl')\cdot\bsDelta_r)}^2.
\]

Setting $\bsL = \bsl + \bsl'$ and defining
\[
\widehat{f}_*(\bsL) := \sum_{\bsl' \in P_{N,\boldsymbol{z}_r}^\perp} \widehat{f}(\bsh+\bsL-\bsl')\, \widehat{f}(\bsh+\bsl'),
\]
we can rewrite the double sum as
\[
\sum_{\bsl,\bsl' \in P_{N,\boldsymbol{z}_r}^\perp}
\widehat{f}(\bsh+\bsl)\, \widehat{f}(\bsh+\bsl')\, e^{2\pi i (\bsl+\bsl')\cdot \bsDelta_r} 
= \sum_{\bsL \in P_{N,\boldsymbol{z}_r}^\perp}\widehat{f}_*(\bsL)\exp(2\pi\icomp \bsL\cdot\bsDelta_r).
\]
Finally, integrating over $\bsDelta_r \in [0,1)^d$ and using the orthonormality of the Fourier basis gives
\[
      \int_{[0,1)^d} |\widehat{f}_{N,\bsz_r,\bsDelta_r}(\bsh)|^4\rd\bsDelta_r
      =\int_{[0,1)^d} \abs{\sum_{\bsL \in P_{N,\boldsymbol{z}_r}^\perp}\widehat{f}_*(\bsL)\exp(2\pi\icomp \bsL\cdot\bsDelta_r)}^2\rd\bsDelta_r
      = \sum_{\bsL \in P_{N,\boldsymbol{z}_r}^\perp} |\widehat{f}_*(\bsL)|^2. \qedhere
\]
\end{proof}

\begin{lemma}\label{lem:2ndmomentbound} 
For any $r\in\{1{:}R\}$ and $\bsh\in \mathcal{A}_d(N/2)$, define
\begin{align*}
    r_{\min}(\bsh) & :=\min_{\bsl\in P_{N,\boldsymbol{z}_r}^\perp\setminus \{\bszero\}}r_{2\alpha, \bsgamma}(\bsh+\bsl),\\
    r_{\mathrm{sum}}(\bsh) & :=\sum_{\bsl\in P_{N,\boldsymbol{z}_r}^\perp\setminus \{\bszero\}} \frac{1}{r^2_{2\alpha, \bsgamma}(\bsh+\bsl)},\\
    S(\bsh) & :=\left(\sum_{\bsl\in P_{N,\boldsymbol{z}_r}^\perp\setminus \{\bszero\} }|\widehat{f}(\bsh+\bsl)|^2r_{2\alpha, \bsgamma}(\bsh+\bsl)\right)^{1/2}.
\end{align*}
Then we have
\[ 
    \sigma^2(\bsh):=\int_{[0,1)^d} |\widehat{f}_{N,\bsz_r,\bsDelta_r}(\bsh)|^4\rd\bsDelta_r-\left(\int_{[0,1)^d} |\widehat{f}_{N,\bsz_r,\bsDelta_r}(\bsh)|^2\rd\bsDelta_r\right)^2 \leq  \frac{8|\widehat{f}(\bsh)|^2}{r_{\min}(\bsh)}S(\bsh)^2+2r_{\mathrm{sum}}(\bsh)S(\bsh)^4.
\]
\end{lemma}
\begin{proof}
    We first bound the $\bsl = \bszero$ summand in \eqref{eqn:2ndmoment}:
    \[
    \left|\sum_{\bsl'\in P_{N,\boldsymbol{z}_r}^\perp}\widehat{f}(\bsh+\bsl')\widehat{f}(\bsh-\bsl')\right|\leq \sum_{\bsl'\in P_{N,\boldsymbol{z}_r}^\perp}\frac{|\widehat{f}(\bsh+\bsl')|^2+|\widehat{f}(\bsh-\bsl')|^2}{2}=\sum_{\bsl' \in P_{N,\boldsymbol{z}_r}^\perp}|\widehat{f}(\bsh+\bsl')|^2.
    \]
    Then, by \eqref{eqn:1stmoment} and \eqref{eqn:2ndmoment},
    \begin{equation}\label{eqn:varbound}
        \sigma^2(\bsh)
    \leq \sum_{\bsl \in P_{N,\boldsymbol{z}_r}^\perp\setminus \{\bszero\}}\abs{\sum_{\bsl'\in P_{N,\boldsymbol{z}_r}^\perp}\widehat{f}(\bsh+\bsl')\widehat{f}(\bsh+\bsl-\bsl')}^2.
    \end{equation}
    
    For any $\bsl \in P_{N,\boldsymbol{z}_r}^\perp\setminus \{\bszero\}$, we can decompose
    \[ 
    \sum_{\bsl'\in P_{N,\boldsymbol{z}_r}^\perp}\widehat{f}(\bsh+\bsl')\widehat{f}(\bsh+\bsl-\bsl')=2\widehat{f}(\bsh)\widehat{f}(\bsh+\bsl)+\sum_{\bsl'\in P_{N,\boldsymbol{z}_r}^\perp\setminus\{\bszero,\bsl\} }\widehat{f}(\bsh+\bsl')\widehat{f}(\bsh+\bsl-\bsl').
    \]
    By the Cauchy--Schwarz inequality,
    \begin{align}\label{eqn:fhfhl}
       & \abs{\sum_{\bsl'\in P_{N,\boldsymbol{z}_r}^\perp}\widehat{f}(\bsh+\bsl')\widehat{f}(\bsh+\bsl-\bsl')}^2 \nonumber \\
       & = \abs{2\widehat{f}(\bsh)\widehat{f}(\bsh+\bsl)+\sum_{\bsl'\in P_{N,\boldsymbol{z}_r}^\perp\setminus\{\bszero,\bsl\} }\widehat{f}(\bsh+\bsl')\widehat{f}(\bsh+\bsl-\bsl')}^2\nonumber  \\
       & \leq 8|\widehat{f}(\bsh)|^2|\widehat{f}(\bsh+\bsl)|^2+2\abs{\sum_{\bsl'\in P_{N,\boldsymbol{z}_r}^\perp\setminus\{\bszero,\bsl\} }\widehat{f}(\bsh+\bsl')\widehat{f}(\bsh+\bsl-\bsl')}^2.
    \end{align}
    The second term can be further bounded by
    \begin{align*}
        & \abs{\sum_{\bsl'\in P_{N,\boldsymbol{z}_r}^\perp\setminus\{\bszero,\bsl\} }\widehat{f}(\bsh+\bsl')\widehat{f}(\bsh+\bsl-\bsl')}^2 \\
        & \leq \left(\sum_{\bsl'\in P_{N,\boldsymbol{z}_r}^\perp\setminus\{\bszero,\bsl\} }\frac{1}{r_{2\alpha, \bsgamma}(\bsh+\bsl')r_{2\alpha, \bsgamma}(\bsh+\bsl-\bsl')}\right) \\
        & \quad \times  \left(\sum_{\bsl'\in P_{N,\boldsymbol{z}_r}^\perp\setminus\{\bszero,\bsl\} }|\widehat{f}(\bsh+\bsl')|^2|\widehat{f}(\bsh+\bsl-\bsl')|^2r_{2\alpha, \bsgamma}(\bsh+\bsl')r_{2\alpha, \bsgamma}(\bsh+\bsl-\bsl')\right).
    \end{align*}
    Since we have 
    \begin{align*}
        &\sum_{\bsl'\in P_{N,\boldsymbol{z}_r}^\perp\setminus\{\bszero,\bsl\} }\frac{1}{r_{2\alpha, \bsgamma}(\bsh+\bsl')r_{2\alpha, \bsgamma}(\bsh+\bsl-\bsl')} \\
    & \leq \sum_{\bsl'\in P_{N,\boldsymbol{z}_r}^\perp\setminus\{\bszero,\bsl\} }\frac{1}{2r^2_{2\alpha, \bsgamma}(\bsh+\bsl')}+\sum_{\bsl'\in P_{N,\boldsymbol{z}_r}^\perp\setminus\{\bszero,\bsl\} }\frac{1}{2r^2_{2\alpha, \bsgamma}(\bsh+\bsl-\bsl')} \leq r_{\mathrm{sum}}(\bsh),
    \end{align*}
    for any $\bsl \in P_{N,\boldsymbol{z}_r}^\perp\setminus \{\bszero\}$, we obtain
    \begin{align*}
        & \sum_{\bsl \in P_{N,\boldsymbol{z}_r}^\perp\setminus \{\bszero\}}\abs{\sum_{\bsl'\in P_{N,\boldsymbol{z}_r}^\perp\setminus\{\bszero,\bsl\} }\widehat{f}(\bsh+\bsl')\widehat{f}(\bsh+\bsl-\bsl')}^2 \\
        & \leq r_{\mathrm{sum}}(\bsh)\sum_{\bsl \in P_{N,\boldsymbol{z}_r}^\perp\setminus \{\bszero\}}\,\,\sum_{\bsl'\in P_{N,\boldsymbol{z}_r}^\perp\setminus\{\bszero,\bsl\} }|\widehat{f}(\bsh+\bsl')|^2|\widehat{f}(\bsh+\bsl-\bsl')|^2r_{2\alpha, \bsgamma}(\bsh+\bsl')r_{2\alpha, \bsgamma}(\bsh+\bsl-\bsl')\\
        & = r_{\mathrm{sum}}(\bsh)\sum_{\bsl' \in P_{N,\boldsymbol{z}_r}^\perp\setminus \{\bszero\}}|\widehat{f}(\bsh+\bsl')|^2 r_{2\alpha, \bsgamma}(\bsh+\bsl')\sum_{\bsl\in P_{N,\boldsymbol{z}_r}^\perp\setminus\{\bszero,\bsl'\} }|\widehat{f}(\bsh+\bsl-\bsl')|^2r_{2\alpha, \bsgamma}(\bsh+\bsl-\bsl')\\
        & \leq r_{\mathrm{sum}}(\bsh)\sum_{\bsl' \in P_{N,\boldsymbol{z}_r}^\perp\setminus \{\bszero\}}|\widehat{f}(\bsh+\bsl')|^2 r_{2\alpha, \bsgamma}(\bsh+\bsl') S(\bsh)^2 = r_{\mathrm{sum}}(\bsh) S(\bsh)^4.
    \end{align*}
    
    It follows from \eqref{eqn:fhfhl} and the above bound that
    \begin{align*}
        \sum_{\bsl \in P_{N,\boldsymbol{z}_r}^\perp\setminus \{\bszero\}}\abs{\sum_{\bsl'\in P_{N,\boldsymbol{z}_r}^\perp}\widehat{f}(\bsh+\bsl')\widehat{f}(\bsh+\bsl-\bsl')}^2
        & \leq 8|\widehat{f}(\bsh)|^2\sum_{\bsl \in P_{N,\boldsymbol{z}_r}^\perp\setminus \{\bszero\}}|\widehat{f}(\bsh+\bsl)|^2+2r_{\mathrm{sum}}(\bsh) S(\bsh)^4 \\
        & \leq \frac{8|\widehat{f}(\bsh)|^2}{r_{\min}(\bsh)} S(\bsh)^2 +2r_{\mathrm{sum}}(\bsh) S(\bsh)^4,
    \end{align*}
    which completes the proof because of \eqref{eqn:varbound}.
\end{proof}

\begin{lemma}\label{lem:goodlattice}
Let $\bsh\in \Acal_d (N/2)$. Define
\begin{equation}\label{eq:def_Sh}
   \Bar{S}(\bsh) := 4\left(
\frac{\|f\|^2_{d,\alpha,\bsgamma}}{N-1} + 
\sum_{\bsl \in N \cdot \ZZ^d \setminus \{\bszero\}} |\widehat{f}(\bsh+\bsl)|^2 r_{2\alpha,\bsgamma}(\bsh+\bsl)
\right)^{1/2}. 
\end{equation}
Then, for any $f \in \Hcal_{d,\alpha,\bsgamma}$, the following bounds hold:
\begin{equation}\label{eqn:Shbound}
  \Pr\left(S(\bsh)^2\ge\Bar{S}(\bsh)^2\right)\leq \frac{1}{16}   
\end{equation}
and 
\[ 
\Pr\left(\{r_{\min}(\bsh)< N^{2\alpha}_*\}\cup \{r_{\mathrm{sum}}(\bsh)\ge \Bar{r}_{\mathrm{sum}}\}\right)\leq \frac{1}{8},
\]
where $N_*$ and $\Bar{r}_{\mathrm{sum}}$ are as defined in \eqref{eq:N_star_def}  and \eqref{eq:def_r_sum}, respectively.
\end{lemma}
\begin{proof}
First, we compute
\begin{align*}
   & \frac{1}{(N-1)^d}\sum_{\bsz_r\in \{1{:}(N-1)\}^d}S(\bsh)^2 \\
   & =\frac{1}{(N-1)^d}\sum_{\bsz_r\in \{1{:}(N-1)\}^d}\sum_{\bsl\in P_{N,\boldsymbol{z}_r}^\perp\setminus \{\bszero\}}|\widehat{f}(\bsh+\bsl)|^2 r_{2\alpha,\bsgamma}(\bsh+\bsl)\\
   & = \sum_{\bsl \in \ZZ^d\setminus (N\cdot \ZZ^d)}|\widehat{f}(\bsh+\bsl)|^2r_{2\alpha, \bsgamma}(\bsh+\bsl)\left(\frac{1}{(N-1)^d}\sum_{\bsz_r\in \{1{:}(N-1)\}^d}\bsone\{\bsl \in P_{N,\boldsymbol{z}_r}^\perp\}\right) \\
   & \quad +\sum_{\bsl \in N\cdot \ZZ^d\setminus \{\bszero\} }|\widehat{f}(\bsh+\bsl)|^2 r_{2\alpha, \bsgamma}(\bsh+\bsl) \\
   & \le \frac{\|f\|^2_{d,\alpha,\bsgamma}}{N-1 }+\sum_{\bsl \in N\cdot \ZZ^d\setminus \{\bszero\} }|\widehat{f}(\bsh+\bsl)|^2r_{2\alpha, \bsgamma}(\bsh+\bsl),
\end{align*}
where we have used the bound \eqref{eqn:ellinP}. Then, the first statement \eqref{eqn:Shbound} follows from Markov's inequality.

Next, since $N_* < N/2$ and $r_{2\alpha,\bsgamma}(\bsh+\bsl) > (N/2)^{2\alpha}$ for 
$\bsl \in N\cdot \ZZ^d\setminus \{\bszero\}$, we have 
$r_{\min}(\bsh) < N_*^{2\alpha}$ if and only if 
$P_{N,\boldsymbol{z}_r}^\perp \cap \mathcal{B}_{\bsh} \neq \emptyset$. 
Hence, $\Pr(r_{\min}(\bsh) < N_*^{2\alpha}) \leq 1/16$ by Lemma~\ref{lem:nointersect}.

Finally, to bound $\Pr(r_{\mathrm{sum}}(\bsh) \geq \Bar{r}_{\mathrm{sum}})$, define 
\[
R_{\bsh}(L)
:= |\{\bsl \in P_{N,\boldsymbol{z}_r}^\perp\setminus (N\cdot \ZZ^d):
      r_{2\alpha,\bsgamma}(\bsh+\bsl) \leq L^{2\alpha}\}|.
\]
When $r_{\min}(\bsh) \geq N_*^{2\alpha}$, 
integration by parts for the Riemann--Stieltjes integral yields
    \begin{align*}
    \sum_{\bsl\in P_{N,\boldsymbol{z}_r}^\perp\setminus (N\cdot \ZZ^d) } \frac{1}{r^2_{2\alpha, \bsgamma}(\bsh+\bsl)}
    & = \int_{N_*}^\infty L^{-4\alpha}\rd R_{\bsh}(L) = R_{\bsh}(L) L^{-4\alpha}\bigg\vert_{N_*}^\infty -
    \int_{N_*}^\infty R_{\bsh}(L) \rd L^{-4\alpha} \\
    & = 4\alpha \int_{N_*}^\infty R_{\bsh}(L) L^{-4\alpha-1} \rd L,
    \end{align*}
    where, in the last equality, we have used the fact that $R_{\bsh}(L)\leq |\Acal_{d,\alpha,\bsgamma}(L)|$ is dominated by $L^{4\alpha}$. 
    Since it follows from \eqref{eqn:ellinP} that
    \[
        \e[R_{\bsh}(L)] = \frac{1}{N-1}\, |\{\bsl \in \ZZ^d\setminus (N\cdot \ZZ^d): r_{2\alpha,\bsgamma}(\bsh+\bsl)\leq L^{2\alpha}\}| \leq \frac{|\Acal_{d,\alpha,\bsgamma}(L)|}{N-1},
    \]
    we obtain
    \begin{align*}
        \e[r_{\mathrm{sum}}(\bsh)\bsone\{r_{\min}(\bsh)> N^{2\alpha}_*\}]\leq &\int_{N_*}^\infty \frac{|\Acal_{d,\alpha,\bsgamma}(L)|}{N-1} 4\alpha L^{-4\alpha-1} \rd L+\sum_{\bsl\in N\cdot \ZZ^d\setminus\{\bszero\}}\frac{1}{r^2_{2\alpha, \bsgamma}(\bsh+\bsl)}. 
    \end{align*}
    
    It remains to bound the second term on the right-hand side. 
    Since $\bsh=(h_1,\dots,h_d)\in \Acal_d(N/2)$, we have $|h_j|\le N/2$ for all $j$.  
    For $\bsl\in N\cdot \ZZ^d$, write $\ell_j = k_jN$ with $k_j\in\ZZ$.  
    If $k_j=0$ then $|\ell_j+h_j| \geq 0$, 
    while if $k_j\neq 0$ then
    \[
        |\ell_j+h_j| \geq (|k_j|-1)N + \frac{N}{2},
    \]
    so in particular $\ell_j+h_j \neq 0$ when $k_j\neq 0$.
    Writing $(\bsl_{\fraku},\bszero_{-\fraku})$ 
    for the vector with components $\ell_j$ if $j\in\fraku$ and zero otherwise, the above observation implies that
    \[ \mathrm{supp}(\bsh+(\bsl_{\fraku},\bszero_{-\fraku})) \supseteq \fraku, \]
    for any $\bsh\in \Acal_d(N/2)$ and $\bsl_{\fraku}\in (N\cdot \ZZ\setminus \{0\})^{|\fraku|}$. Thus, we have
    \begin{align*}
    \sum_{\bsl\in N\cdot \ZZ^d\setminus\{\bszero\}}\frac{1}{r^2_{2\alpha, \bsgamma}(\bsh+\bsl)} 
    & = 
    \sum_{\emptyset \neq \fraku \subseteq \{1{:}d\}}
    \sum_{\bsl_{\fraku}\in (N\cdot \ZZ\setminus\{\bszero\})^{\abs{\fraku}}}\frac{1}{r^2_{2\alpha, \bsgamma}(\bsh+(\bsl_{\fraku},\bszero_{-\fraku}))}\\
    &= \sum_{\emptyset \neq \fraku \subseteq \{1{:}d\}} 
     \sum_{\bsl_{\fraku}\in (N\cdot \ZZ\setminus\{\bszero\})^{\abs{\fraku}}} 
     \gamma^{2}_{\mathrm{supp}(\bsh+(\bsl_{\fraku},\bszero_{-\fraku}))}
    \abs{\bsh+(\bsl_{\fraku},\bszero_{-\fraku})}^{-4\alpha}\\
    &\le \sum_{\emptyset \neq \fraku \subseteq \{1{:}d\}} \gamma^{2}_{\fraku}
    \sum_{\bsl_{\fraku}\in (N\cdot \ZZ\setminus\{\bszero\})^{\abs{\fraku}}}
    \prod_{j\in\fraku}(h_j + \ell_j)^{-4\alpha}\\
    &\le \sum_{\emptyset \neq \fraku \subseteq \{1{:}d\}} \gamma^{2}_{\fraku}
    \left(2\sum_{k=1}^\infty ((k-1)N+N/2)^{-4\alpha}\right)^{\abs{\fraku}}\\
    &= \sum_{\emptyset \neq \fraku \subseteq \{1{:}d\}} \gamma^{2}_{\fraku}
    \left(\frac{2^{1+4\alpha}}{N^{4\alpha}}\sum_{k=1}^\infty (2k-1)^{-4\alpha}\right)^{\abs{\fraku}}\\
    &\le \frac{1}{N^{4\alpha}} 
    \sum_{\emptyset \neq \fraku \subseteq \{1{:}d\}} \gamma^{2}_{\fraku}
    \left(2^{1+4\alpha}\zeta(4\alpha)\right)^{\abs{\fraku}},
    \end{align*}
    where we have used the assumption that $\gamma_{\mathfrak{v}}\le\gamma_{\mathfrak{w}}$ 
    whenever $\mathfrak{w}\subseteq \mathfrak{v}$. Combining the above expectation bound with Markov’s inequality, and recalling that $\Pr(r_{\min}(\bsh)< N^{2\alpha}_*)\leq 1/16$, we obtain the desired probability estimate.
\end{proof}

Now we are ready to show that, with high probability, every $\bsh\notin K'_N$ has a small value of 
\(\operatorname{median}_{r\in\{1{:}R\}} \lvert \widehat{f}_{N,\bsz_r,\bsDelta_r}(\bsh)\rvert^2\), 
while every \(\bsh'\in K'_N\) with a large value of \(\lvert\widehat{f}(\bsh')\rvert^2\) 
also attains a correspondingly large 
\(\operatorname{median}_{r\in\{1{:}R\}} \lvert \widehat{f}_{N,\bsz_r,\bsDelta_r}(\bsh')\rvert^2\).
These two facts together imply that, except with small probability, 
the ordering of frequencies induced by the medians in Step 2 of Algorithm~\ref{alg:median} 
correctly separates the significant Fourier coefficients (which should belong to \(K_N\)) 
from the insignificant ones (which should lie outside \(K_N\)); 
this is the key ingredient for controlling the third summand in \eqref{eqn:errordecom}.

\begin{lemma}\label{lem:mediannorm}
    For $\bsh\in \mathcal{A}_d(N/2)$ and $f\in \Hcal_{d,\alpha,\bsgamma}$, let
\[
\delta'(\bsh) = 4 \left(\frac{\|f\|^2_{d,\alpha,\bsgamma}}{N_*^{2\alpha}(N-1)} + F(\bsh)\right)^{1/2},
\]
where $F(\bsh)$ is defined in \eqref{eq:def_F_bsh'}. Then 
\[
\Pr\left( \operatorname*{median}_{r\in\{1{:}R\}}
   \left| \widehat{f}_{N,\bsz_r,\bsDelta_r}(\bsh)\right|^2 
   > \delta'(\bsh)^2 \right)\leq \frac{1}{2}2^{-R/2}.
\]
\end{lemma}
\begin{proof}
By replacing \eqref{eqn:deltamean} with \eqref{eqn:1stmoment}, the same argument as described in the proof of Lemma~\ref{lem:concentration} yields
\begin{align*}
    \e\left[|\widehat{f}_{N,\bsz_r,\bsDelta_r}(\bsh)|^2\bsone\{P_{N,\boldsymbol{z}_r}^\perp\cap \mathcal{B}_{\bsh}= \emptyset\}\right]
    \leq  \frac{\|f\|^2_{d,\alpha,\bsgamma}}{N_*^{2\alpha}(N-1) }+\sum_{\bsl \in N\cdot \ZZ^d }|\widehat{f}(\bsh+\bsl)|^2=\frac{\delta'(\bsh)^2}{16}.
\end{align*}
Hence, using Lemma~\ref{lem:nointersect} together with Markov’s inequality (as in \eqref{eqn:Markov}), we obtain
\begin{align*}
   & \Pr\left( | \widehat{f}_{N,\bsz_r,\bsDelta_r}(\bsh)|^2 > \delta'(\bsh)^2 \right) \\
   & \leq \Pr(P_{N,\boldsymbol{z}_r}^\perp\cap \mathcal{B}_{\bsh}\neq \emptyset)+\frac{1}{\delta'(\bsh)^2}\e\left[|\widehat{f}_{N,\bsz_r,\bsDelta_r}(\bsh)|^2\bsone\{P_{N,\boldsymbol{z}_r}^\perp\cap \mathcal{B}_{\bsh}= \emptyset\}\right] \le \frac{1}{8}.
\end{align*}
Finally, applying Lemma~\ref{lem:mediantrick} to the $R$ independent replicates gives the claimed probability bound.
\end{proof}

The next lemma provides an upper bound on $F(\bsh'_N)$, the smallest value of $F(\bsh')$ among $\bsh'\in K'_N\subset \mathcal{A}_d(N/2)$, in terms of the norm $\|f\|_{d,\alpha,\bsgamma}$.

\begin{lemma}\label{lem:FNbound}
    For any $f\in \Hcal_{d,\alpha,\bsgamma}$,
    \[
      F(\bsh'_N)\leq \frac{16}{15}\,\frac{\|f\|^2_{d,\alpha,\bsgamma}}{N_*^{2\alpha}(N-1)}.
    \]
\end{lemma}
\begin{proof}
First observe that for any $\bsh\in (-N/2,N/2)^d$,
\[
   r_{2\alpha, \bsgamma}(\bsh)=\min_{\bsl\in N\cdot \ZZ^d} r_{2\alpha, \bsgamma}(\bsh+\bsl).
\]
Hence,
\begin{align}\label{eqn:FNbound}
    \|f\|^2_{d,\alpha,\bsgamma}=\sum_{\bsh\in\ZZ^d} |\widehat{f}({\bsh})|^2 r_{2\alpha,\bsgamma}(\bsh)\geq \sum_{\bsh\in \mathcal{A}_d(N/2)} F(\bsh)  r_{2\alpha,\bsgamma}(\bsh)\geq F(\bsh'_N)\sum_{\bsh'\in K'_N}r_{2\alpha,\bsgamma}(\bsh').
\end{align}
Since $|K'_N|=N$, while at most $(N-1)/16$ indices $\bsh'\in \mathcal{A}_d(N/2)$ satisfy
$r_{2\alpha,\bsgamma}(\bsh') < N_*^{2\alpha}$ (by~\eqref{eqn:Nstar}), we obtain
\[
   \sum_{\bsh'\in K'_N}r_{2\alpha,\bsgamma}(\bsh')
   \ge N_*^{2\alpha}\left(N - \frac{N-1}{16}\right)
   \ge \frac{15}{16} N_*^{2\alpha}(N-1).
\]
Inserting this estimate into~\eqref{eqn:FNbound} yields the desired inequality.
\end{proof}

\begin{corollary}\label{cor:nonKN}
For $f\in \Hcal_{d,\alpha,\bsgamma}$, define
\[
   \Bar{\delta}=4\left(\frac{31}{15}\,\frac{\|f\|^2_{d,\alpha,\bsgamma}}{N_*^{2\alpha}(N-1)}\right)^{1/2}.
\]
Then for all $\bsh\in \mathcal{A}_d(N/2)\setminus K'_N$,
\[
   \Pr\left( \operatorname*{median}_{r\in\{1{:}R\}}
     \left| \widehat{f}_{N,\bsz_r,\bsDelta_r}(\bsh)\right|^2
     > \Bar{\delta}^2 \right) \le \frac{1}{2}\,2^{-R/2}.
\]
\end{corollary}
\begin{proof}
If $\bsh\in \mathcal{A}_d(N/2)\setminus K'_N$, then by Lemma~\ref{lem:FNbound} and the definition of $K'_N$
\[
   F(\bsh) \leq F(\bsh'_N)
   \leq \frac{16}{15}\,\frac{\|f\|^2_{d,\alpha,\bsgamma}}{N_*^{2\alpha}(N-1)}.
\]
The claimed probability estimate then follows immediately from Lemma~\ref{lem:mediannorm}.
\end{proof}

\begin{lemma}\label{thm:inKN}
For  $\bsh\in \mathcal{A}_d(N/2)$ and $f\in \Hcal_{d,\alpha,\bsgamma}$, if 
    \begin{equation}\label{eqn:largefhat}
       |\widehat{f}(\bsh)|^2\ge \Bar{S}(\bsh)^2\max\left(960 N^{-2\alpha}_*,\sqrt{240 \Bar{r}_{\mathrm{sum}} }\right),
    \end{equation}
    where $\Bar{S}(\bsh)$ and $\Bar{r}_{\mathrm{sum}}$ are as defined in \eqref{eq:def_Sh} and \eqref{eq:def_r_sum}, respectively, then 
    \[
    \Pr\left(\operatorname*{median}_{r\in\{1{:}R\}} | \widehat{f}_{N,\bsz_r,\bsDelta_r}(\bsh)|^2\leq \frac{F(\bsh)}{2}\right)\le \frac{1}{2}\left(\frac{3}{4}\right)^{R/2}.
    \]
\end{lemma}
\begin{proof}
From \eqref{eqn:1stmoment}, we have
\[
\int_{[0,1)^d}|\widehat{f}_{N,\bsz_r,\bsDelta_r}(\bsh)|^2\rd\bsDelta_r=\sum_{\bsl \in P_{N,\boldsymbol{z}_r}^\perp}|\widehat{f}(\bsh+\bsl)|^2\geq F(\bsh)\geq |\widehat{f}(\bsh)|^2,
\]
which provides a lower bound on the expectation of $|\widehat{f}_{N,\bsz_r,\bsDelta_r}(\bsh)|^2$ when $\bsz_r$ is fixed and $\bsDelta_r$ follows a uniform distribution over $[0,1)^d$. 
Let $\sigma^2(\bsh)$ denote the corresponding variance as in Lemma~\ref{lem:2ndmomentbound}. 
Then, Cantelli's inequality gives
\begin{align}\label{eqn:Cantelli}
     & \int_{[0,1)^d}\bsone\left\{|\widehat{f}_{N,\bsz_r,\bsDelta_r}(\bsh)|^2\leq \frac{F(\bsh)}{2}\right\}\rd\bsDelta_r \notag\\
     & \le \int_{[0,1)^d}\bsone\left\{|\widehat{f}_{N,\bsz_r,\bsDelta_r}(\bsh)|^2\leq\int_{[0,1)^d} |\widehat{f}_{N,\bsz_r,\bsDelta_r}(\bsh)|^2\rd\bsDelta_r - \frac{F(\bsh)}{2}\right\}\rd\bsDelta_r \notag\\
    & \le \frac{\sigma(\bsh)^2}{(F(\bsh)/2)^2+\sigma(\bsh)^2} \le \frac{1}{1+|\widehat{f}(\bsh)|^4/(4\sigma(\bsh)^2)}.    
\end{align}

    Under the ``good lattice'' event, i.e., when $r_{\min}(\bsh)> N^{2\alpha}_*$, $r_{\mathrm{sum}}(\bsh)< \Bar{r}_{\mathrm{sum}}$, and $S(\bsh)^2< \Bar{S}(\bsh)^2$, Lemma~\ref{lem:2ndmomentbound} gives
    \[ 
    \frac{\sigma^2(\bsh)}{|\widehat{f}(\bsh)|^4}\leq 8N^{-2\alpha}_*\frac{\Bar{S}(\bsh)^2}{|\widehat{f}(\bsh)|^2}  +2\Bar{r}_{\mathrm{sum}}  \frac{\Bar{S}^4(\bsh)}{|\widehat{f}(\bsh)|^4}.
    \]
    By condition~\eqref{eqn:largefhat}, the right-hand side is at most $1/60$, so the right-most side of \eqref{eqn:Cantelli} is bounded by $1/16$. Hence, by Lemma~\ref{lem:goodlattice},
    \begin{align*}
     & \Pr\left(\operatorname*{median}_{r\in\{1{:}R\}}| \widehat{f}_{N,\bsz_r,\bsDelta_r}(\bsh)|^2\leq \frac{F(\bsh)}{2}\right) \\
     & \le \Pr\left(\{r_{\min}(\bsh) \leq N^{2\alpha}_*\}\cup \{r_{\mathrm{sum}}(\bsh)\ge \Bar{r}_{\mathrm{sum}}\}\cup \{S(\bsh)^2\ge\Bar{S}(\bsh)^2\}\right) \\
     & \quad + \Pr(\{|\widehat{f}_{N,\bsz_r,\bsDelta_r}(\bsh)|^2\leq F(\bsh)/2\}\cap \{r_{\min}(\bsh)> N^{2\alpha}_*\}\cap \{r_{\mathrm{sum}}(\bsh)< \Bar{r}_{\mathrm{sum}}\}\cap \{S(\bsh)^2<\Bar{S}(\bsh)^2\})\\
     & \le \frac{1}{4}.
    \end{align*}
    The conclusion then follows from Lemma~\ref{lem:mediantrick}.
\end{proof}

Finally, we are now at the stage of proving a probabilistic upper bound on the third summand in the error decomposition \eqref{eqn:errordecom}.
The following theorem formalizes this bound.

\begin{theorem}\label{thm:K'minusK}
    Let odd prime $N$ and odd natural number $R$ be sufficiently large such that
    \begin{equation}\label{eq:choose_R_second}
    \varepsilon_2:=\frac{|\mathcal{A}_d(N/2)|}{2}\left(\frac{1}{2}\right)^{R/2}+\frac{N}{2}\left(\frac{3}{4}\right)^{R/2 }< 1.
    \end{equation}
    Then, for any $f\in \Hcal_{d,\alpha,\bsgamma}$, with probability at least $1-\varepsilon_2$,
    \[ 
    \sum_{\bsh'\in K'_N\setminus K_N}|\widehat{f}(\bsh')|^2\leq \frac{16\|f\|^2_{d,\alpha,\bsgamma}}{N_*^{2\alpha}}  \left(\frac{62N}{15(N-1)}+\frac{2N-1}{N-1 }\max\left(960,\sqrt{240 N^{4\alpha}_*\Bar{r}_{\mathrm{sum}}}\right) \right).
    \]
\end{theorem}
\begin{proof}
Applying Corollary~\ref{cor:nonKN} to all $\bsh\in \mathcal{A}_d(N/2)\setminus K'_N$ and Lemma~\ref{thm:inKN} to all $\bsh'\in K'_N$ satisfying \eqref{eqn:largefhat}, a union bound argument shows that, with probability at least $1-\varepsilon_2$, we have:
    \begin{itemize}
        \item For all $\bsh\notin K'_N$,
        \[
            \operatorname*{median}_{r\in\{1{:}R\}}| \widehat{f}_{N,\bsz_r,\bsDelta_r}(\bsh)|^2\leq \Bar{\delta}^2.
        \]
        \item For all $\bsh'\in K'_N$ satisfying \eqref{eqn:largefhat},
        \[ 
            \operatorname*{median}_{r\in\{1{:}R\}}| \widehat{f}_{N,\bsz_r,\bsDelta_r}(\bsh')|^2> \frac{F(\bsh')}{2}.
        \]
    \end{itemize}

    Now, for any $\bsh'\in K'_N\setminus K_N$, either
    \[
    |\widehat{f}(\bsh')|^2< \Bar{S}(\bsh)^2\max\left(960 N^{-2\alpha}_*,\sqrt{240 \Bar{r}_{\mathrm{sum}} }\right)
    \]
    or, if this is not the case, we argue as follows: for any $\bsh\in K_N\setminus K'_N$, by the definition of $K_N$,
    \[
    \frac{F(\bsh')}{2} <
    \operatorname*{median}_{r\in\{1{:}R\}}| \widehat{f}_{N,\bsz_r,\bsDelta_r}(\bsh')|^2 \le \operatorname*{median}_{r\in\{1{:}R\}}| \widehat{f}_{N,\bsz_r,\bsDelta_r}(\bsh)|^2\leq \Bar{\delta}^2,
    \]
    which, together with the definition of $F(\bsh')$, implies
    \[
    |\widehat{f}(\bsh')|^2\leq F(\bsh')\leq 2\Bar{\delta}^2.
    \]
    Combining the upper bounds in the two cases, we obtain
    \[
        \sum_{\bsh'\in K'_N\setminus K_N}|\widehat{f}(\bsh')|^2
        \leq |K'_N\setminus K_N|\cdot 2\Bar{\delta}^2
        + \max\left(960 N^{-2\alpha}_*,\sqrt{240 \Bar{r}_{\mathrm{sum}} }\right)\sum_{\bsh'\in K'_N\setminus K_N} \Bar{S}^2(\bsh').
    \]
    Finally, using $|K'_N\setminus K_N|\le N$ and 
    \begin{align*}
    \sum_{\bsh'\in K'_N\setminus K_N} \Bar{S}^2(\bsh')
    & \leq 16\left(|K'_N\setminus K_N|\frac{\|f\|^2_{d,\alpha,\bsgamma}}{N-1 }+\sum_{\bsh'\in K'_N\setminus K_N}\sum_{\bsl \in N\cdot \ZZ^d\setminus \{\bszero\} }|\widehat{f}(\bsh'+\bsl)|^2r_{2\alpha, \bsgamma}(\bsh'+\bsl)\right)\\
    & \leq 16 \left( \frac{N\|f\|^2_{d,\alpha,\bsgamma}}{N-1 }+\|f\|^2_{d,\alpha,\bsgamma}\right),
    \end{align*}
    the claim follows.
\end{proof}

\section{Numerical experiments}\label{sec:experiment}
Here, we conduct numerical experiments to support our theoretical findings. The key feature of Algorithm~\ref{alg:median} is its universality; it does not require any prior information on the smoothness $\alpha$ and the weights $\bsgamma$ as an input, yet it automatically exploits the regularity of the target function. The focus of these experiments is to test this capability to exploit smoothness. To do this, we consider the simplest multivariate setting, i.e., the unweighted, bivariate case.

Following a previous paper by the same authors on median lattice-based algorithms \cite{PKG24}, we test the following two periodic functions, which were also used in \cite{CGK24}:
\begin{align*}
f_1(\bsx) &=  \prod_{j=1}^d \frac{121\sqrt{33}}{100}\max \left\{\frac{25}{121}-\left(x_j-\frac{1}{2} \right)^2, 0 \right \}, \\
f_2(\bsx) &= \prod_{j=1}^{d} \left(x_j-\frac{1}{2} \right)^2\sin(2\pi x_j-\pi),
\end{align*}
with $d=2$. It is important to note that these functions belong to different smoothness classes: $f_1\in \Hcal_{d,3/2-\epsilon,\bsgamma}$ and $f_2\in \Hcal_{d,5/2-\epsilon,\bsgamma}$ for any arbitrarily small $\epsilon>0$.
This means that the smoothness of $f_1$ is arbitrarily close to $3/2$ and that of $f_2$ is arbitrarily close to $5/2$.
Consequently, based on Theorem~\ref{thm:combined_error} and Corollary~\ref{cor:L_2_bound}, our universal median lattice-based algorithm is expected to achieve $L_2$ error rates arbitrarily close to $M^{-3/2}$ for $f_1$ and $M^{-5/2}$ for $f_2$, respectively, where $M=NR$ denotes the total number of function evaluations.

For our implementation of Algorithm~\ref{alg:median}, we set $R$ to the smallest odd integer that satisfies the condition in \eqref{eq:R_large_enough} with dimension $d=2$ and failure probability $\varepsilon=0.01$. This choice ensures that the error bound shown in Corollary~\ref{cor:L_2_bound} holds with a probability of at least $0.99$.
For comparison, we also run the median lattice-based algorithm from \cite{PKG24}. Unlike our present method, this algorithm requires prior knowledge of the function's smoothness and the weights to construct a proper index set. For a fair comparison, we use the same parameter settings and failure probability ($\varepsilon = 0.01$) as reported in the original paper \cite{PKG24}.

Figure~\ref{fig:result1} presents the experimental results, plotting the $L_2$ error of both algorithms against the total number of function evaluations, $M$, on a log-log scale. The left and right panels show the results for test functions $f_1$ and $f_2$, respectively. The $L_2$ error is calculated exactly by using the true Fourier coefficients and the $L_2$-norm of the functions, see \cite{CGK24,PKG24}. Each panel includes reference lines for the rates $M^{-\alpha/2}$, $M^{-3\alpha/4}$, and the optimal rate $M^{-\alpha}$, where the smoothness parameter is $\alpha=3/2$ for $f_1$ and $\alpha=5/2$ for $f_2$. For both functions, our universal algorithm significantly outperforms the non-universal baseline. This superior performance can be attributed to its ability to adaptively construct the index set $K_N$ based on the magnitude of the estimated Fourier coefficients. Nevertheless, the observed convergence rate for our algorithm is approximately $M^{-3\alpha/4}$ within the tested range of $M$, see below for an explanation for this phenomenon.

\begin{figure}[htbp]
    \centering
    \includegraphics[width=0.45\linewidth]{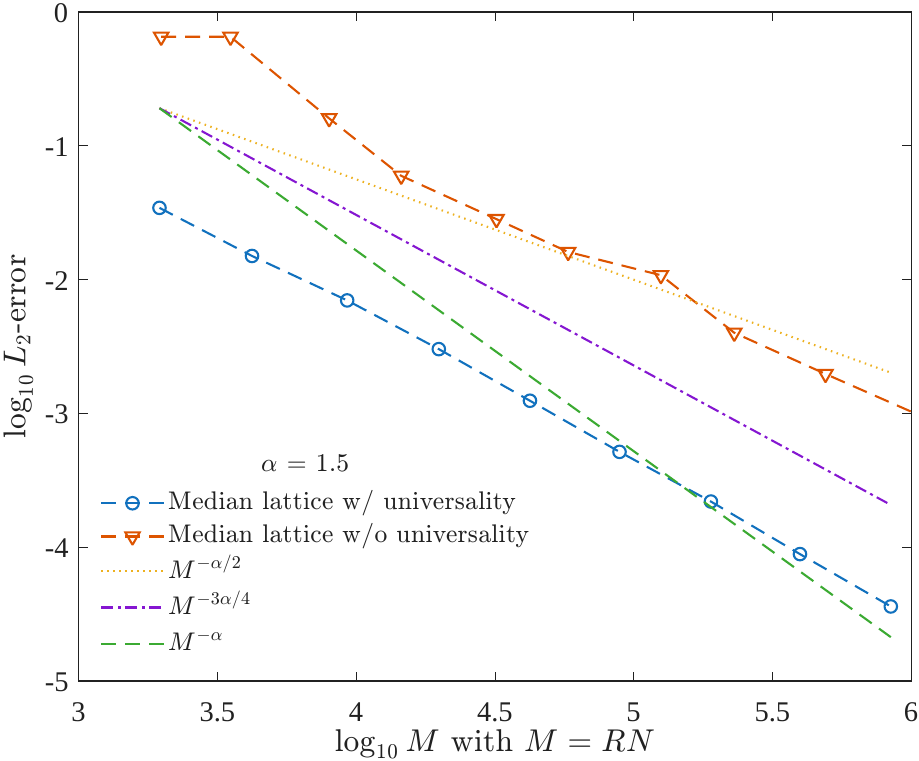}\ \ 
    \includegraphics[width=0.45\linewidth]{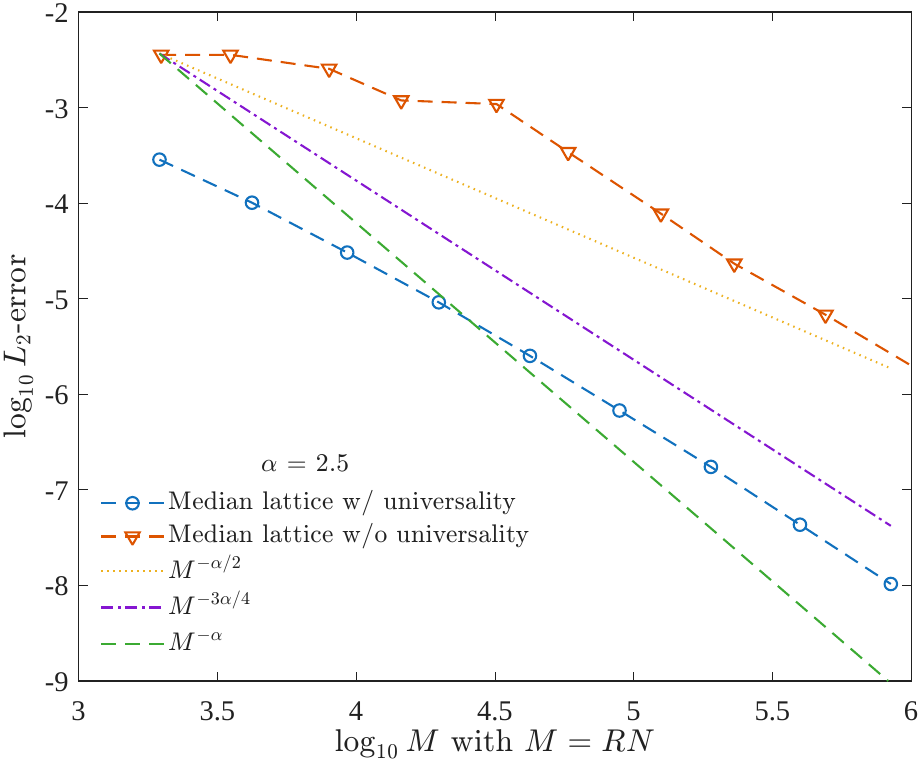}\ \ 
    \caption{$L_2$-approximation error for the test functions $f_1$ (left) and $f_2$ (right) as a function of the total number of function evaluations $M$.}
    \label{fig:result1}
\end{figure}

To evaluate how well our universal algorithm selects the index set $K_N$, we compute the normalized truncation error:
\[ \left(\frac{\sum_{\bsk\in \ZZ^d\setminus K_N}|\widehat{f}(\bsk)|^2}{\|f\|^2_{L_2}}\right)^{1/2}=\left(1-\frac{\sum_{\bsk\in K_N}|\widehat{f}(\bsk)|^2}{\|f\|^2_{L_2}}\right)^{1/2}, \]
where $K_N$ is the set of $N$ indices constructed in the second step of Algorithm~\ref{alg:median}, and $\widehat{f}(\bsk)$ denotes the true $\bsk$-th Fourier coefficient.
We benchmark this error against the theoretically optimal performance achievable with any set of $N$ indices. This benchmark is given by the oracle truncation error:
\[ \left(1-\frac{\sum_{\bsk\in K^{\mathrm{oracle}}_N}|\widehat{f}(\bsk)|^2}{\|f\|^2_{L_2}}\right)^{1/2}. \]
Here, the oracle set $K^{\mathrm{oracle}}_N$ is defined as the set of indices corresponding to the $N$ Fourier coefficients with the largest magnitudes, that is, for any $\bsk\in K^{\mathrm{oracle}}_N$ and $\bsl\notin K^{\mathrm{oracle}}_N$, we have $|\widehat{f}(\bsk)|\ge |\widehat{f}(\bsl)|$.

Figure~\ref{fig:result2} compares these two truncation errors, plotting them against the total number of function evaluations, $M$, on a log-log scale. Although the index set $K_N$ constructed by our algorithm is stochastic, the figure shows the result from a single realization. For both test functions, the curves from our algorithm and the oracle are nearly indistinguishable, which implies that the adaptive choice of $K_N$ is remarkably close to optimal. The convergence rate of this truncation error is again observed to be approximately $M^{-3\alpha/4}$ within the tested range. This finding explains the performance seen in Figure~\ref{fig:result1}: the overall $L_2$ error is already limited by the convergence of the truncation error itself. We therefore conclude that the observed suboptimal convergence rate is a finite-sample effect, and achieving the theoretical rate of $M^{-\alpha}$ would likely require experiments in a larger-$M$ regime.

\begin{figure}[htbp]
    \centering
    \includegraphics[width=0.45\linewidth]{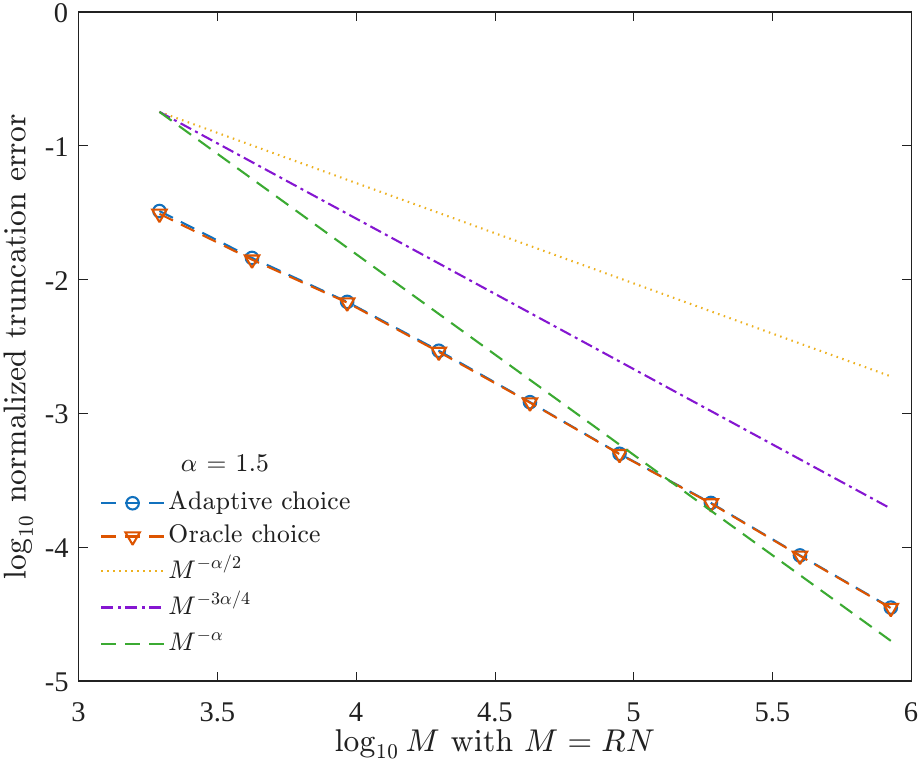}\ \ 
    \includegraphics[width=0.45\linewidth]{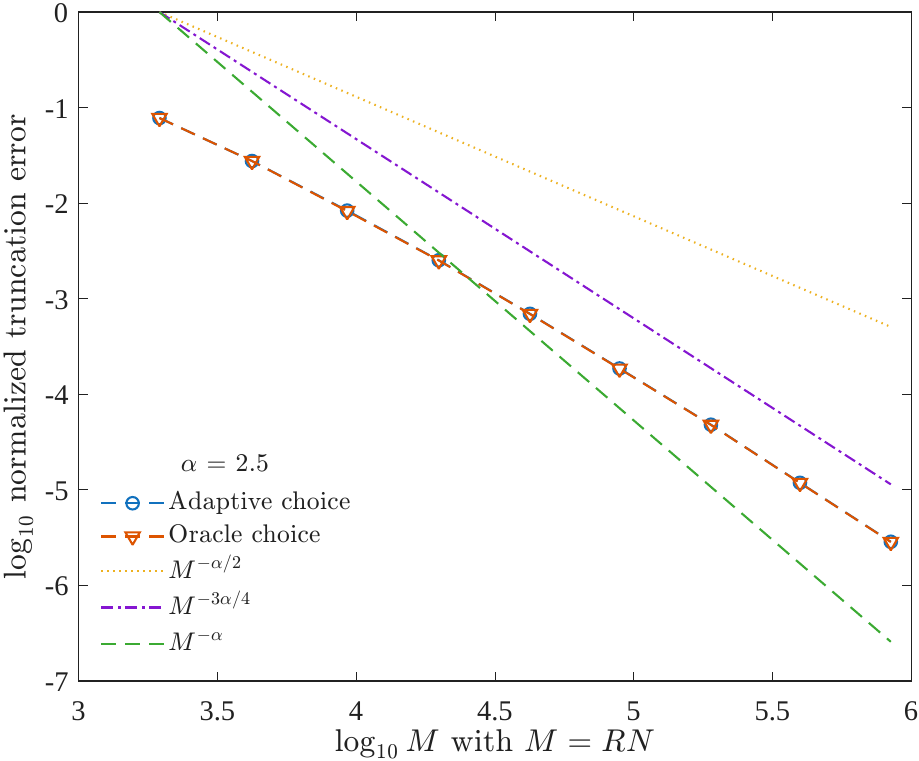}\ \ 
    \caption{Normalized truncation error for the test functions $f_1$ (left) and $f_2$ (right). The error from the adaptive selection of the index set is compared to the optimal (oracle) error.}
    \label{fig:result2}
\end{figure}

\section{Discussion}\label{sec:discussion}

This work extends the median lattice algorithm from \cite{PKG24} to the Korobov space setting with unknown smoothness $\alpha$ and weight parameters $\bsgamma$. A primary computational bottleneck, as pointed out in Remark~\ref{rem:evaluation_vs_work}, lies in Step 2 of Algorithm~\ref{alg:median}, where we we must identify the $N$ largest values of $\mathrm{median}_{r\in\{1{:}R\}}|\widehat{f}_{N,\bsz_r,\bsDelta_r}(\bsh)|^2$ across the candidate set $\Acal_{d}(N/2)$, whose cardinality grows exponentially with the dimension $d$. A naive implementation of this step is therefore computationally prohibitive in high dimensions.

This motivates us to ask whether $K_N$, or a suitable substitute, can be found in computation time polynomial in $d$. A natural solution is to constrain the search to a manageable subset $\Acal^*\subseteq \Acal_{d}(N/2)$, designed within a given computational budget. The design of $\Acal^*$ can incorporate prior knowledge of $\alpha$ and $\bsgamma$ if available. Indeed, with full knowledge, we recover the original algorithm from \cite{PKG24} by setting $\Acal^*$ to be $\Acal_d(N_*)$ defined therein.  In the absence of such knowledge, one could empirically estimate the parameters by fitting the decay of $\operatorname*{median}_{r\in\{1{:}R\}}| \widehat{f}_{N,\bsz_r,\bsDelta_r}(\bsh)|^2$ on a small test set to the inverse weights $r^{-1}_{2\alpha,\bsgamma}(\bsh)$. A full investigation of this adaptive approach is beyond the scope of this paper, so we reserve it for future work.

A second research direction involves generalizing our method to general Sobolev spaces, thereby removing the periodic boundary conditions inherent to Korobov spaces. A straightforward generalization replaces trigonometric series with finite Walsh series and lattice rules with polynomial lattice rules \cite{dick:pill:2010}. While this allows for efficient computation via the Fast Walsh–Hadamard Transform, the resulting approximants are generally discontinuous and lie outside the target Sobolev space. Consequently, they exhibit visual artifacts compared to smoother alternatives like kernel interpolation. A central open question is whether our strategy can be adapted to produce smooth approximants, a challenge we leave for future study.

\section*{Acknowledgments}

The first and third authors acknowledge the support of the Austrian Science Fund (FWF) Project  P 34808/Grant DOI: 10.55776/P34808. The second author acknowledges the support of JSPS KAKENHI Grant Number 23K03210. For open access purposes, the authors have applied a CC BY public copyright license to any author accepted manuscript version arising from this submission. Moreover, the authors would like to thank D.~Krieg for valuable comments.

\bibliographystyle{amsplain}

\begin{thebibliography}{10}

\bibitem{BGKS24}
F.~Bartel, A.D. Gilbert, F.Y. Kuo, and I.H. Sloan, \emph{Minimal subsampled rank-1 lattices for multivariate approximation with optimal convergence rate}, arXiv preprint arXiv:2506.07729 (2025).

\bibitem{byrenheid2017tight}
G.~Byrenheid, L.~Kämmerer, T.~Ullrich, and T.~Volkmer, \emph{Tight error bounds for rank-1 lattice sampling in spaces of hybrid mixed smoothness}, Numer. Math. \textbf{136} (2017), 993–1034.

\bibitem{CGK24}
M.~Cai, T.~Goda, and Y.~Kazashi, \emph{{$L_2$}-approximation using randomized lattice algorithms}, arXiv preprint arXiv:2409.18757 (2024).

\bibitem{CM17}
A.~Cohen and F.~Migliorati, \emph{{Optimal weighted least squares methods}}, SMAI J. Comput. Math. \textbf{3} (2017), 181--203.

\bibitem{CKNS20}
R.~Cools, F.Y. Kuo, D.~Nuyens, and I.H. Sloan, \emph{Lattice algorithms for multivariate approximation in periodic spaces with general weight parameters}, Contemp. Math. \textbf{754} (2020), 93--113.

\bibitem{DTU18}
D.~D\~ung, V.~Temlyakov, and T.~Ullrich, \emph{Hyperbolic cross approximation}, Advanced Courses in Mathematics. CRM Barcelona, Birkh\"auser/Springer, Cham, 2018.

\bibitem{dick2022component}
J.~Dick, T.~Goda, and K.~Suzuki, \emph{Component-by-component construction of randomized rank-1 lattice rules achieving almost the optimal randomized error rate}, Math. Comp. \textbf{91} (2022), no.~338, 2771--2801.

\bibitem{dick2022lattice}
J.~Dick, P.~Kritzer, and F.~Pillichshammer, \emph{{Lattice Rules: Numerical Integration, Approximation, and Discrepancy}}, Springer Series in Computational Mathematics, vol.~58, Springer Nature Switzerland AG, 2022.

\bibitem{kuo2013high}
J.~Dick, F.~Y. Kuo, and I.~H. Sloan, \emph{High-dimensional integration: the quasi-{M}onte {C}arlo way}, Acta Numer. \textbf{22} (2013), 133--288.

\bibitem{dick:pill:2010}
J.~Dick and F.~Pillichshammer, \emph{{Digital Nets and Sequences: Discrepancy Theory and Quasi–Monte Carlo Integration}}, Cambridge University Press, 2010.

\bibitem{DC24}
M.~Dolbeault and M.A. Chkifa, \emph{{Randomized least-squares with minimal oversampling and interpolation in general spaces}}, SIAM J. Numer. Anal. \textbf{62} (2024), no.~4, 1515--1538.

\bibitem{goda2022free}
T.~Goda and P.~L'Ecuyer, \emph{Construction-free median quasi-monte carlo rules for function spaces with unspecified smoothness and general weights}, SIAM J. Sci. Comput. \textbf{44} (2022), no.~4, A2765--A2788.

\bibitem{goda2024universal}
T.~Goda, K.~Suzuki, and Matsumoto M., \emph{A universal median quasi-{M}onte {C}arlo integration}, SIAM J. Numer. Anal. \textbf{62} (2024), no.~1, 533--566.

\bibitem{kaemmerer2019constructing}
L.~K\"{a}mmerer, \emph{Constructing spatial discretizations for sparse multivariate trigonometric polynomials that allow for a fast discrete {F}ourier transform}, Appl. Comput. Harmon. Anal. \textbf{47} (2019), no.~3, 702--729.

\bibitem{kammerer2019approximation}
L.~K{\"a}mmerer and T.~Volkmer, \emph{Approximation of multivariate periodic functions based on sampling along multiple rank-1 lattices}, J. Approx. Theory \textbf{246} (2019), 1--27.

\bibitem{K19}
D.~Krieg, \emph{{Optimal Monte Carlo methods for $L^2$-approximation}}, Constr. Approx. \textbf{49} (2019), 385--403.

\bibitem{kunsch2019solvable}
R.~J. Kunsch, E.~Novak, and D.~Rudolf, \emph{Solvable integration problems and optimal sample size selection}, Journal of Complexity \textbf{53} (2019), 40--67.

\bibitem{kunsch2019optimal}
R.~J. Kunsch and D.~Rudolf, \emph{Optimal confidence for {M}onte {C}arlo integration of smooth functions}, Adv. Comput. Math. \textbf{45} (2019), no.~5-6, 3095--3122.

\bibitem{kuo2006lattice}
F.~Y. Kuo, I.~H. Sloan, and H.~Wo{\'{z}}niakowski, \emph{Lattice rules for multivariate approximation in the worst case setting}, Monte Carlo and Quasi-Monte Carlo Methods 2004, Springer Berlin Heidelberg, 2006, pp.~289--330.

\bibitem{li2003trigonometric}
D.~Li and F.~J. Hickernell, \emph{Trigonometric spectral collocation methods on lattices}, Recent {A}dvances in {S}cientific {C}omputing and {P}artial {D}ifferential {Equations} (S.~Y. Cheng, C.-W. Shu, and T.~Tang, eds.), American Mathematical Society, Providence, RI, 2003, pp.~121--132.

\bibitem{niemiro2009fixed}
W.~Niemiro and P.~Pokarowski, \emph{Fixed precision {MCMC} estimation by median of products of averages}, Journal of applied probability \textbf{46} (2009), no.~2, 309--329.

\bibitem{novak2008tractability}
E.~Novak and H.~Wo{\'z}niakowski, \emph{{Tractability of Multivariate Problems, Volume I: Linear Information}}, EMS Tracts in Mathematics, vol.~6, EMS, 2008.

\bibitem{PKG24}
Z.~Pan, P.~Kritzer, and T.~Goda, \emph{{$L_2$}-approximation using median lattice algorithms}, arXiv preprint arXiv:2501.15331 (2025).

\bibitem{pan2024super-pol}
Z.~Pan and A.B. Owen, \emph{{Super-polynomial accuracy of multidimensional randomized nets using the median-of-means}}, Math. Comp. \textbf{93} (2024), no.~349, 2265--2289.

\bibitem{sloan1994lattice}
I.~H. Sloan and S.~Joe, \emph{Lattice methods for multiple integration}, Oxford University Press, 1994.

\bibitem{sloan1998when}
I.~H. Sloan and H.~Wo{\'z}niakowski, \emph{{When are quasi-Monte Carlo algorithms efficient for high dimensional integrals?}}, J. Complexity \textbf{14} (1998), no.~1, 1--33.

\bibitem{WW07}
G.~W. Wasilkowski and H.~Wo\'{z}niakowski, \emph{The power of standard information for multivariate approximation in the randomized setting}, Math. Comp. \textbf{76} (2007), no.~258, 965--988.

\bibitem{zeng2006error}
X.~Zeng, K.-T. Leung, and F.~J. Hickernell, \emph{Error analysis of splines for periodic problems using lattice designs}, Monte Carlo and Quasi-Monte Carlo Methods 2004, Springer Berlin Heidelberg, 2006, pp.~501--514.

\end{thebibliography}

\iffalse
\begin{small}

\noindent\textbf{Authors' addresses:}\\

\noindent Zexin Pan\\
Institute of Fundamental and Transdisciplinary Research\\
Zhejiang University\\
866 Yuhangtang Road, Xihu District, Hangzhou, Zhejiang Province, 310058, China\\
\texttt{zep002@zju.edu.cn}\\

\noindent Takashi Goda\\
Graduate School of Engineering\\
The University of Tokyo\\
7-3-1 Hongo, Bunkyo-ku, Tokyo 113-8656, Japan\\
\texttt{goda@frcer.t.u-tokyo.ac.jp}\\

\noindent Peter Kritzer\\
Johann Radon Institute for Computational and Applied Mathematics (RICAM)\\
Austrian Academy of Sciences\\
Altenbergerstr. 69, 4040 Linz, Austria\\
 \texttt{peter.kritzer@oeaw.ac.at}\\

\end{small}
\fi

\end{document}